\documentclass[11pt,a4paper]{amsart}
\usepackage{amssymb,amsmath,epsfig,graphics,mathrsfs,mathtools}
\usepackage{graphicx}
\usepackage{fancyhdr}
\usepackage{microtype}
\usepackage{enumerate}
\fancyheadoffset{0px}
\usepackage{hyperref}
\hypersetup{
 colorlinks   = true,
 urlcolor     = blue,
 linkcolor    = blue,
 citecolor   = red ,
 bookmarksopen=true
}

\newtheorem{theorema}{Theorem}

\newtheorem{thrm}{Theorem}[section]
\newtheorem{lemma}[thrm]{Lemma}
\newtheorem{prop}[thrm]{Proposition}
\newtheorem{cor}[thrm]{Corollary}
\theoremstyle{remark}
\newtheorem{rmrk}[thrm]{Remark}
\theoremstyle{definition}
\newtheorem{dfn}[thrm]{Definition}

\usepackage[margin=3cm]{geometry}
\newcommand{\Q}{\mathbb{Q}}

\newcommand{\Rn}{\mathbb R^n}

\newcommand{\R}{\mathbb R}
\newcommand{\Om}{\Omega}

\newcommand{\p}{\partial}

\newcommand{\la}{\lambda}
\newcommand{\vf}{\varphi}

\newcommand{\G}{\Gamma}

\newcommand{\BB}{\B_r}
\newcommand{\La}{\mathscr L_a}
\newcommand{\Gb}{\overline{\mathscr G}_a}

\renewcommand{\phi}{\varphi}
\renewcommand{\epsilon}{\varepsilon}

\numberwithin{equation}{section}

\newcommand{\beq}{\begin{equation}}
\newcommand{\bea}[1]{\begin{array}{#1} }
\newcommand{\eeq}{ \end{equation}}
\newcommand{\ea}{ \end{array}}

\newcommand{\ve}{\varepsilon}
\newcommand{\Rnn}{\mathbb R^{n+1}}
\newcommand{\Rnp}{\mathbb R^{n+1}_+}

\newcommand{\B}{\mathbb{B}}

\newcommand{\Sa}{\mathbb{S}}
\newcommand{\pa}{\p^a_y v}
\newcommand{\pax}{\p^a_y v(x,0)}
\newcommand{\Wa}{\mathscr W}

\begin{document}

\title[The regular free boundary, etc.]{The regular free boundary in the thin obstacle problem for degenerate parabolic equations}

\author{Agnid Banerjee}
\address{TIFR CAM, Bangalore-560065}
\email[Agnid Banerjee]{agnidban@gmail.com}
\thanks{The first author was supported in part by SERB Matrix grant MTR/2018/000267}

\author{Donatella Danielli}
\address{Department of Mathematics \\ Purdue University \\ 47907 West Lafayette, IN}
\email[Donatella Danielli]{danielli@math.purdue.edu}

\author{Nicola Garofalo}
\address{Dipartimento di Ingegneria Civile, Edile e Ambientale (DICEA) \\ Universit\`a di Padova\\ 35131 Padova, ITALY}
\email[Nicola Garofalo]{rembrandt54@gmail.com}
\thanks{The third author was supported in part by a Progetto SID (Investimento Strategico di Dipartimento) ``Non-local operators in geometry and in free boundary problems, and their connection with the applied sciences'', University of Padova, 2017.}

\author{Arshak Petrosyan}
\address{Department of Mathematics \\ Purdue University\\ 47907 West Lafayette, IN}
\email[Arshak Petrosyan]{arshak@purdue.edu}
\thanks{The fourth author was supported in part by NSF Grant DMS-1800527.}

\dedicatory{Dedicated to Nina, with affection and deep admiration. Her pioneering ideas have left a permanent mark in PDEs, and inspired scores of mathematicians}

\begin{abstract}
In this paper we study the existence, the optimal regularity of solutions, and the regularity of the free boundary near the so-called \emph{regular points} in a thin obstacle problem that arises as the local extension of the obstacle problem for the fractional heat operator $(\partial_t - \Delta_x)^s$ for $s \in (0,1)$. Our regularity estimates are completely local in nature. This aspect is of crucial importance in our forthcoming work on the blowup analysis of the free boundary, including the study of the singular set. Our approach is based on first establishing the boundedness of the time-derivative of the solution. This allows reduction to an elliptic problem at every fixed time level. Using several results from the elliptic theory, including the epiperimetric inequality, we establish the optimal regularity of solutions as well as $H^{1+\gamma,\frac{1+\gamma}{2}}$ regularity of the free boundary near such regular points. 
\end{abstract}

\maketitle

\tableofcontents

\section{Introduction}\label{S:Intro}
\subsection{Statement of the problem}
The primary objective of the present paper is the study of the thin obstacle problem for a degenerate parabolic
operator
\begin{equation}\label{eq:La}
\mathscr L_a U\overset{\rm def}{=}|y|^a \p_t U -\operatorname{div}_{X}(|y|^a
\nabla_{X} U),\quad a\in (-1,1),
\end{equation}
where $X=(x,y)\in\R^{n}\times\R=\R^{n+1}$ and $t\in\R$. For a domain $\mathbb{D}$ in
$\R^{n+1}$, symmetric in $y$, let $\mathbb{D}^\pm=\mathbb{D}\cap \{\pm
y>0\}$, and $D=\mathbb{D}\cap \{y=0\}$. Given a function $\psi$ on
$D\times[t_0,T]$, $t_0<T$, known as \emph{the thin obstacle}, consider the
problem of finding a function $U$ in $\mathbb{D}^+\times(t_0,T]$ such that 
\begin{align}\label{epb}
&\begin{cases}
\mathscr L_a U=0&\text{in}\ \mathbb{D}^+\times(t_0,T],
\\
\min\{U(x,0,t)-\psi(x,t),-\p_y^a
U(x,0,t)\}=0 & \text{on }D\times(t_0,T],\\
\end{cases}
\end{align}
where $\partial_y^a U(x,0,t)$ is the weighted partial derivative of $U$ in
$y$ on $\{y=0\}$, defined by
\begin{equation*}
  \partial_y^a U(x,0,t)\overset{\rm def}{=}\lim_{y\to 0^+} y^a\partial_y U(x,y,t).
\end{equation*}
We also impose initial and lateral boundary conditions
\begin{align}
\label{epb-bdry}
\begin{cases}
U(X,t_0)= \phi_0(X) &\text{on $\mathbb{D}^+$},
\\
U= g &\text{on $(\partial \mathbb{D})^+\times(t_0,T)$},
\end{cases}
\end{align}
obeying the compatibility conditions $\phi_0=g(\cdot , t_0)$ on
$(\p\mathbb{D})^+$, $\phi_0 \geq \psi(\cdot, t_0)$ on $D$, $g \geq \psi$ on
$\partial D\times(t_0,T)$.

The conditions on $D\times(t_0,T]$ in \eqref{epb} are known as the
\emph{Signorini complementarity} (or \emph{ambiguous})
\emph{conditions}. Essentially, $D\times (t_0,T]$ is divided into two regions
\begin{align*}
\Lambda(U)&\overset{\rm def}{=}\{U=\psi\}\cap (D\times (t_0,T])\quad \text{(\emph{coincidence set})},\\
\Omega(U)&\overset{\rm def}{=}\{U>\psi\}\cap (D\times (t_0,T]),
\end{align*}
where the Dirichlet condition
$U(x,0,t)=\psi(x,t)$, and the Neumann-type condition
$\partial_y^aU(x,0,t)=0$ are respectively satisfied.
These two regions are separated by the set 
$$
\Gamma(U)\overset{\rm def}{=}\partial \Lambda(U)\cap
(D\times (t_0,T])\quad \text{(\emph{free boundary})},
$$
which is apriori unknown and may in principle have a complicated structure. Thus
knowing the regularity properties of $\Gamma(U)$ is one of the primary
objectives of the problem.

This type of boundary conditions go back to problems of 
Signorini type from elastostatics, see e.g.\ \cite{DL}, and typically arise in
problems with unilateral constraints. In fact, if 
\begin{multline*}
\mathscr{K}=\mathscr{K}_{\psi,g,\phi_0,\mathbb{D}^+,t_0,T}=\{v\in
\mathscr V_a(\mathbb{D}^+,t_0,T)\mid \\\text{$v\geq \psi$ on $D\times(t_0,T]$, $v=g$ on
  $(\partial \mathbb{D})^+\times(t_0,T]$, $v(\cdot,t_0)=\phi_0$}\}
\end{multline*}
is the constraint set of functions staying above the thin
obstacle $\psi$ (see Section~\ref{S:Notation} for the definition of
spaces $\mathscr V_a$), then we say that $U\in\mathscr K$, with $\partial_t U\in L^2(\mathbb{D}^+\times(t_0,T],|y|^a
  dXdt)$, is a weak solution of 
\eqref{epb}--\eqref{epb-bdry} if it satisfies for a.e.\ $t \in (t_0, T]$ the variational inequality 
\begin{equation}\label{var-ineq}
\int_{\mathbb{D}^+} \partial_t U (v-U)|y|^a
  dX+\int_{\mathbb{D}^+}\langle\nabla_X U,\nabla_X (v-U)\rangle
  |y|^a dX\geq
  0,\quad\text{for any } v\in\mathscr K.
\end{equation}
An important motivation for studying the problem \eqref{epb} is that it
serves as a localization of the (nonlocal) obstacle problem for the fractional heat
equation. More precisely, if $u$ is a function on
$\R^n\times(-\infty,T]$ with a sufficient decay at infinity,
satisfying the obstacle problem
\begin{equation}\label{frac-heat-obt}
\min\{u-\psi, (\partial_t -\Delta_x)^s u\}=0\quad\text{in }D\times(t_0,T],
\end{equation}
with fractional power $s=(1-a)/2\in(0,1)$, consider the solution $U$ of the \emph{extension problem}
\begin{equation*}
\begin{cases}
  \mathscr L_a U =0&\text{in }\R^{n+1}_+\times (-\infty,T],\\
  U(x,0,t)=u(x,t)  &\text{on }\R^n\times (-\infty,T],
\end{cases}
\end{equation*}
vanishing at infinity. Then, representing the equation $\mathscr L_a U=0$ in nondivergence form as
\begin{equation*}
{\p_t U} - \Delta_x U = \mathscr B_a U,
\end{equation*}  
where $\mathscr B_a = \frac{\p^2}{\p y^2} + \frac ay \frac{\p }{\p y}$
is the generator of the Bessel semigroup on $(\R^+,y^a dy)$, it can be
shown that
\begin{equation*}
(\p_t - \Delta_x)^s u(x,t)=- C_s\partial_y^a U(x,0,t), \quad\text{with
}C_s=\frac{2^{2s-1} \G(s)}{\G(1-s)},
\end{equation*}
see \cite{NS}, \cite{ST}, and also Section~3 in \cite{BG}.
Thus, the obstacle problem \eqref{frac-heat-obt} 
can be written in the form \eqref{epb}, thus localizing the
problem.

We further remark that the differential equation $\mathscr L_a U=0$ is a special case
of the following class of degenerate parabolic equations 
\[
\p_t(\omega(X) U) = \operatorname{div}_X(A(X) \nabla_X U)
\]
first studied by Chiarenza and Serapioni in \cite{CSe}. In that paper the authors assumed that $\omega\in L^1_{\rm loc}(\Rnn)$ is a Muckenhoupt $A_2$-weight independent of the time variable $t$, and that the symmetric matrix-valued function $X\mapsto A(X)$ (also independent of $t$) verifies the following degenerate ellipticity assumption for a.e. $X\in  \Rnn$, and for every $\xi\in \Rnn$:
\[
\la \omega(X) |\xi|^2 \le \langle A(X)\xi,\xi\rangle \le \la^{-1} \omega(X) |\xi|^2,
\]
for some $\la >0$.
Under such hypothesis they established a parabolic strong Harnack inequality, and therefore the local H\"older continuity of the weak solutions. The differential equation in \eqref{epb} is a special case of those treated in \cite{CSe} since, given that $a \in (-1,1)$, the function $\omega(X) = \omega(x,y) = |y|^a$ is an $A_2$-weight in $\Rnn$.

\subsection{Notations}\label{S:Notation}
 We indicate with $x = (x_1,\ldots,x_n)$ a generic point in the
 ``thin'' space $\Rn$. By the letter
 $y$ we will denote the ``extension variable'' on $\R$. The generic
 point in the ``thick'' space $\R^{n+1}=\R^n\times\R$ will be denoted by $X =
 (x,y)$ with $x\in\R^n$, $y\in\R$. Also, in many cases we will identify
 the thin space $\R^n$ with the subset $\{y=0\}=\R^n\times\{0\}\subset\R^{n+1}$.

 The points in the thin space-time
 $\R^n\times\R$ will be denoted by $(x,t)$ and the ones in the thick
 space-time either by $(X,t)$ or $(x,y,t)$.

 Given $r>0$, $x_0\in\R^n$, $X_0\in\R^{n+1}$, $t_0\in\R$ we denote
 \begin{alignat*}{2}
B_r(x_0)&=\{x\in\R^n\mid |x-x_0|<r\} &\quad&\text{(thin Euclidean
  ball)}\\
\B_r(X_0)&=\{X\in\R^{n+1}\mid |X-X_0|<r\} &&\text{(thick Euclidean
  ball)}\\
Q_r(x_0,t_0)&=B_r(x_0)\times(t_0-r^2,t_0]&&\text{(thin parabolic
  cylinder)}\\
\Q_r(X_0,t_0)&=\B_r(X_0)\times(t_0-r^2,t_0]&&\text{(thick parabolic cylinder)}.
 \end{alignat*}
We typically drop centers from the notations above if they coincide
with the origin. Thus, we write $B_r$, $\B_r$, $Q_r$, and $\Q_r$ for
$B_r(0)$, $\B_r(0)$, $Q_r(0,0)$, and $\Q_r(0,0)$, respectively.

For a set $E$ in the thick space $\R^{n+1}$ or the space-time
$\R^{n+1}\times\R$ we denote
$$
E^\pm=E\cap \{\pm y>0\}.
$$
Thus, $\B_r^\pm$, $\Q_r^\pm$ denote thick half-balls and parabolic
half-cylinders, $(\partial\B_r)^+$ is a half-sphere, etc.

Given an open set $\mathbb{D}\subset \Rnn$, we denote by $W^{1,2}(\mathbb{D},|y|^a dX)$ the space of functions in $L^2(\mathbb{D},|y|^a dX)$ whose distributional derivatives of order one belong to $L^2(\mathbb{D},|y|^a dX)$.
We endow such Hilbert space with the norm
\[
\|f\|_{W^{1,2}(\mathbb{D},|y|^a dX)} = \left(\int_{\mathbb{D}} \left(f^2 + |\nabla_X f|^2\right) |y|^a dX\right)^{1/2}.
\]
For given numbers $a\in (-1,1)$, $-\infty \le T_1<T_2 \le \infty$, and an open set $\mathbb{D}\subset \Rnn$, we define
\[
\mathscr V_a(\mathbb{D},T_1,T_2) = L^2((T_1, T_2);W^{1,2}(\mathbb{D},|y|^a dX)),
\]
and equip such space with the norm
\begin{equation}
\|w\|_{\mathscr V_a(\mathbb{D},T_1,T_2)} = \left(\int_{T_1}^{T_2} \int_{\mathbb{D}} \left(w^2+  |\nabla w|^2 \right) |y|^a dX dt\right)^{1/2} < \infty.
\end{equation}

We will also use fairly standard notations for H\"older classes of
functions $C^{\ell}$, $\ell=k+\gamma$, with $k\in\{0\}\cup\mathbb{N}$
and $0<\gamma\leq 1$, as well as their parabolic counterparts
$H^{\ell,\ell/2}$. We refer the reader to \cite{DGPT} for precise definitions.

We will also need the following weighted H\"older classes. For
an open set $\mathbb{D}\subset \subset \Rnn_+$, $K=\overline{\mathbb{D}}$, and $\gamma\in (0,1]$, by $C^{1+\gamma}_a(K)$
we denote the class of functions $f$ such that $\nabla_x f$,
$|y|^a \partial_yf\in C^{\gamma}(K)$, equipped with the norm
\begin{equation}\label{wnorm}
\|f \|_{C_a^{1+\gamma}(K)} = \|f\|_{C^{0}(K)} + \|\nabla_x f\|_{C^{\gamma}(K)} + \||y|^a \partial_yf\|_{C^{\gamma}(K)}.
\end{equation}
The parabolic version of the space above for $K=\overline{\mathbb{D}}\times[T_1,T_2]$ is the space
$H_a^{1+\gamma,(1+\gamma)/2}(K)$ of functions $f$ with
$\nabla_x f$, $|y|^a\partial_y f\in
H^{\gamma,\gamma/2}(K)$, $f\in
C^{(1+\gamma)/2}_t(K)$, with the norm
\begin{equation}\label{wnorm-parab}
\|f
\|_{H_a^{1+\gamma,(1+\gamma)/2}(K)}=\|f\|_{C^{(1+\gamma)/2}_t(K)}+
\|\nabla_x f \|_{H^{\gamma,\gamma/2}(K)}+
\||y|^a\partial_y f \|_{H^{\gamma,\gamma/2}(K)}.
\end{equation}

\subsection{Main results}
In the case $a=0$, the problem \eqref{epb}--\eqref{epb-bdry} is the
parabolic Signorini problem for the standard heat equation with
results on the regularity of solutions going back to the works of
Athanosopoulos \cite{Ath82} and Uraltseva \cite{Ura85} (see also
\cite{AU88}). For a comprehensive treatment of this problem we refer
to the work of three of us with T.~To \cite{DGPT}, as well as to
\cite{PZ}. For an alternative approach to this and related problems
we also refer to \cite{ACM1}.

For the case $a\in (-1,1)$, a version of the problem
\eqref{epb}--\eqref{epb-bdry} has been recently studied in
\cite{ACM}, where the authors used global assumptions on initial
data $\phi_0$ to infer quasi-convexity properties of the solutions,
leading to their optimal regularity, as well as to the regularity
of the free boundary near certain types of points. The approach that we take in
the present paper is purely local, which is of crucial importance in the
further analysis of the 
problem as it allows to consider the blowups at free boundary points,
leading to their fine classification, see our forthcoming paper \cite{BDGP2}.

We now state the main results in this paper. Since we are mainly
interested in local properties of the solutions $U$ of \eqref{epb}, as
well as their free boundaries $\Gamma(U)$, we may assume that the
domain $\mathbb{D}\times(t_0,T]$ is a parabolic cylinder centered on a
thin space-time, and by using a translation and scaling, we may assume
$$
\mathbb{D}\times(t_0,T]=\Q_1.
$$
Throughout this paper for a given number $a\in (-1,1)$ we denote by $\kappa_0$ the number
\begin{equation}\label{mino}
\kappa_0 = \frac{3-a}{2}=1+s,
\end{equation} 
where $s=(1-a)/2\in (0,1)$ is fractional power of the heat equation,
in the corresponding obstacle problem \eqref{frac-heat-obt}.

\begin{theorema}[Regularity of solutions]\label{T:or}
Let $U$ be a weak solution to \textup{\eqref{epb}--\eqref{epb-bdry}} in $\Q_1^+$
in the sense that the variational inequality \eqref{var-ineq} is
satisfied. Assume 
also that $\psi\in H^{4,2}(Q_1)$. Then,
\begin{enumerate}[\upshape (i)]
\item $U\in H_a^{1+\gamma,(1+\gamma)/2}\left(\overline{\Q_{1/2}^+}\right)$ and
  $\partial_t U\in L^\infty(\Q_{1/2}^+)$ for some $\gamma=\gamma(n,a)\in
  (0,1)$ and
$$
 \|U\|_{H_a^{1+\gamma,(1+\gamma)/2}\left(\overline{\Q_{1/2}^+}\right)}+\|\partial_t
U\|_{L^\infty(\Q_{1/2}^+)}\leq C(n,a)\left(\|U\|_{L^2(\Q_1^+,|y|^a dXdt)}+\|\psi\|_{H^{4,2}(Q_1)}\right).
$$
\item $U(\cdot,y,\cdot)\in
H^{\kappa_0,\kappa_0/2}\left(\overline{Q_{1/2}}\right)$ uniformly
for $y\in [0,1/2]$
with
$$
\sup_{y\in[0,1/2]} \|U(\cdot,y,\cdot)\|_{H^{\kappa_0,\kappa_0/2}\left(\overline{Q_{1/2}}\right)}\leq C(n,a)\left(\|U\|_{L^2(\Q_1^+,|y|^a dXdt)}+\|\psi\|_{H^{4,2}(Q_1)}\right).
$$
\end{enumerate}
\end{theorema}

We explicitly observe here that the estimates above are purely local in nature and do not depend on the initial and lateral boundary data
$\phi_0$ and $g$, in contrast to the results in \cite{ACM}. 
We also
note that while the estimate in part (ii) is optimal in $x$ variables,
part (i) gives a joint regularity in $(X,t)$ variables, which is necessary in compactness arguments.

To state our next result, we need to introduce the notion of
\emph{regular free boundary points}. Two equivalent definitions of such points based on parabolic and elliptic Almgren-Poon type 
frequency functions are given in Section~\ref{S:rfb}, see Definitions~\ref{reg} and \ref{D:reg} (as well as Lemma~\ref{equiv} for their equivalence).
In more elementary terms, we say that $(x_0,t_0)\in
\Gamma(U)$ is regular if
\begin{align*}
L_{\rm par}&=\limsup_{r \to 0} \frac{ \|U(\cdot,0,\cdot)-\psi\|_{L^{\infty}(Q_r(x_0, t_0))}}{r^{\kappa_0}},\quad
             \text{or equivalently,}\\
L_{\rm ell}&=\limsup_{r \to 0} \frac{
             \|U(\cdot,0,t_0)-\psi(\cdot,t_0)\|_{L^{\infty}(B_r(x_0))}}{r^{\kappa_0}}
\end{align*}
is bounded away from $0$ and $\infty$. We denote the set of all
regular free boundary points $\Gamma_{\kappa_0}(U)$ and call it the
\emph{regular set}.

\begin{theorema}[Smoothness of the regular set]\label{T:regreg}
Let $U$ be as in Theorem~\ref{T:or}. Then $\Gamma_{\kappa_0}(U)$ is a
relatively open subset of $\Gamma(U)$ and is locally given as a graph
$$
x_n=g(x_1,\ldots,x_{n-1},t)
$$
with $g\in H^{1+\gamma,(1+\gamma)/2}$, after a possible rotation of
coordinate axes in the thin space $\R^n$.
\end{theorema}
Concerning our approach, we stress that the one in the present paper
differs in a significant way from that in \cite{DGPT}. Here, we make
use of a crucial new information, namely that for a solution $U$ of
\eqref{epb} we have that $\partial_t U$ is locally bounded, see
Theorem~\ref{T:or}. We mention that for the case $a=0$ treated in
\cite{DGPT}, this fact was first established by one of us and Zeller
in \cite{PZ}. The boundedness of $\partial_t U$ allows us to consider,
at every fixed time level $t_0\in (-1,0]$, the elliptic problem \eqref{ell} below
for $u(X) = U(X,t_0)$. Once this reduction is made, we use several
results from the elliptic theory to establish our main
results. Primarily, we take advantage of a monotonicity formula of
Almgren type that improves on that in \cite{CSS} (as it allows for
a bounded, rather than  Lipschitz, right-hand side). We also rely on a
Weiss type monotonicity formula as well as the epiperimetric
inequality in \cite{GPPS}.   
 
While in this paper we restrict ourselves to the study of regular free
boundary points, in the forthcoming paper \cite{BDGP2} we take a more
systematic parabolic approach to the classification of free boundary points
based on an Almgren-Poon type monotonicity formula. In that paper we
also prove a structural theorem on the so-called singular set of the
free boundary with the help of Weiss and Monneau type monotonicity formulas.

The structure of the paper is as follows. In Section~\ref{S:exist} we address the question of existence and uniqueness of local solutions to \eqref{epb}. In Section~\ref{lreg} we prove Theorem \ref{T:or}. In Section~\ref{S:rfb} we recall several results from the elliptic theory (such as a truncated Almgren-type monotonicity formula, the monotonicity of a Weiss-type functional, and an epiperimetric inequality), which we use to establish Theorem \ref{T:regreg}. Finally, in the Appendix we collect some auxiliary results needed in the proof of Theorem \ref{T:or}.

\section{Existence and uniqueness of solutions}
\label{S:exist}

In this section, we address the question of existence and uniqueness of local solutions to \eqref{epb}. This is the content of Theorem~3.4 in \cite{ACM}, the proof of which crucially relies on regularity estimates for the corresponding nonlocal parabolic obstacle problem. Here we provide an alternative proof  independent of such nonlocal regularity estimates, and  which can possibly be of independent interest. For the sake of simplicity, throughout this section we will assume $g=0$.

We first note that, by even reflection in the variable $y$, we can extend $U$ to the whole $\Q_1$. Therefore, it suffices to consider questions of existence and uniqueness for the  following problem
\begin{equation}\label{epb22}
 \begin{cases}
\mathscr L_a U = 0 &\text{in}\ \Q_1^+ \cup \Q_1^-,
\\
\min\{U(x,0,t)-\psi(x,t), -\partial_y^a U(x,0,t)\}=0 &\text{on }Q_1,
\\
U(x,-y,t) = U(x,y,t)&\text{in}\   \Q_1,\\
U(X,-1)= \varphi_0(X)&\text{on $\B_1^+$},
\\
U= 0& \text{on $\partial \B_1^+ \times (-1,0)$}.
\end{cases}
\end{equation}
We also assume that the obstacle $\psi$ be compactly supported in $B_1 \times [-1,0]$, and we indicate with $\tilde \psi$ a compactly supported extension of $\psi$ to  $\B_1 \times [-1,0]$ which is  symmetric in $y$.

Throughout the paper we assume that $\psi$ is at least of parabolic
H\"older class $H^{2,1}$ (we will actually  need $H^{4,2}$ for some results). This hypothesis implies that, if 
\[
F(X,t)  \overset{\rm def}{=} -|y|^{-a} \La  \tilde \psi(X,t),
\]
then $F \in L^\infty(\Rnn\times \R)$.

Our first step consists in reducing  \eqref{epb22} to the case of zero obstacle by introducing the function
\begin{equation}\label{zo}
V(X,t)  = U(X,t) - \tilde \psi(X,t).
\end{equation}
Since $U$ solves \eqref{epb22}, we have in $\Q_1^+ \cup \Q_1^-$
\[
\La V = \La U - \La \tilde \psi = |y|^a F.
\]
We thus see that, if we let $\tilde \vf_0 = \vf_0 - \tilde \psi$, then the function $V$ satisfies the following zero obstacle problem
\begin{equation}\label{pb222}
 \begin{cases}
\La V = |y|^a F & \text{in}\ \Q_1^+ \cup \Q_1^-,
\\
\min\{V(x,0,t),-\partial_y^a V(x,0,t)\}=0 & \text{on}\ Q_1,
\\ 
V(x,-y,t) = V(x,y,t) &\text{in}\ \Q_1,
\\
V(X,-1) = \tilde \vf_0(X) &\text{on $\B_1^+$},
\\
V= 0&\text{on $\partial \B_1^+ \times (-1,0)$}.
\end{cases}
\end{equation}

We next establish existence and uniqueness of solutions to the problem \eqref{pb222} by appropriately formulating it in the framework of variational inequalities of evolution, following the approach in \cite{DL}.

In the Hilbert space $\mathscr V_a(\B_1^+,-1,0) = L^2((-1,0);W^{1,2}(\mathbb{B}_1^+, |y|^a dX))$, we introduce the closed convex subset
\begin{multline*}
\mathscr K=\{v\in \mathscr V_a(\B_1^+,-1,0)\mid\\ v\geq 0\text{ on }Q_1, \ v=0 \text{ on } \partial \mathbb{B}_1^+\times (-1, 0),
\ v(X,-1)=\tilde\vf_0(X)\text{ on } \B_1^+\}.
\end{multline*}
In addition, for $v\in\mathscr K$, we define
\begin{equation*}
\Psi(v)=\int_{B_1}\zeta(v(x))\,dx\quad\text{with}\quad \zeta(s)=
\begin{cases}
0,&{s\geq0}\\
+\infty,&{s<0}.
\end{cases}
\end{equation*}
Assume we are given $F\in L^\infty(\mathbb{Q}_1^+)$, with $\partial_t F\in L^\infty(\mathbb{Q}_1^+)$. We say that $V$ is a weak solution to
\begin{equation}\label{epb333}
 \begin{cases}
\La V = |y|^a F\qquad&\text{in }\mathbb{Q}_1^+ ,\\
\min\{V(x,0,t),-\partial_y^a V(x,t)\}=0 &\text{on }Q_1,\\
V(X,-1)=\tilde\vf_0(X)&\text{on } \B_1^+,\\
V=0\ &\text{on $\partial \B_1^+ \times (-1, 0)$},
\end{cases}
\end{equation}
if $V\in\mathscr K$, $\partial_t V\in L^2(\Q_1^+, |y|^adXdt)$ and
it satisfies for a.e.\ $t$ the following variational inequality
\begin{equation}\label{epb333b}
\begin{multlined}
\int_{\mathbb{B}_1^+} \partial_t V(v -V)\,|y|^a dX+\int_{\mathbb{B}_1^+}
\langle \nabla_X V,\nabla_X(v -V)\rangle
|y|^adX+\Psi(v)-\Psi(V)\\\geq\int_{\mathbb{B}_1^+} F(v-V)\ |y|^adX,
\end{multlined}
\end{equation}
for all $v\in \mathscr V_a(\B_1^+, -1, 0)$ such that $v=0$ on $(\partial \B_1)^+\times(-1,0)$.

We approximate \eqref{epb333} with the following penalization problem
\begin{equation}\label{pen1}
\begin{cases}
\La V_\ve = |y|^a F_{\ve},&\text{in}\ \mathbb{Q}_1^+,
\\
\partial_y^a V_\epsilon(x,0,t)=\beta_\ve(V_\ve),&\text{for}\ (x,t)\in Q_1,
\\
V_\ve(X,-1)=\tilde\vf_0(X)&\text{on } \B_1^+,
\\
V_\ve=0\ &\text{on $\partial \B_1^+ \times (-1, 0)$},
\end{cases}
\end{equation}
where $F_{\ve}$  is a mollification  of  $F$  and  the penalty function $\beta_\ve\in C^{0,1}(\R)$ is given by
\begin{equation*}
\beta_\ve(s)=
\begin{cases}
\ve+\frac{s}{\ve},&s\leq -2\ve^2,\\
\frac{s}{2\ve},&-2\ve^2<s<0,\\
0,\qquad &s\geq 0.
\end{cases}
\end{equation*}
Clearly, $V_\ve$ is a solution to \eqref{pen1} if, and only if, it satisfies for a.e.\ $t$
\begin{equation}\label{varp}
\begin{multlined}
\int_{\mathbb{B}_1^+} \partial_t V_\ve(v -V_\ve)\ |y|^a
dX+\int_{\mathbb{B}_1^+} \langle\nabla_X V_\ve,\nabla_X(v -V_\ve)\rangle
|y|^adX+\Psi_\ve(v)-\Psi_\ve(V_\ve)\\\geq\int_{\mathbb{B}_1^+}
F_\ve(v-V_\ve)\ |y|^adX
\end{multlined}
\end{equation}
for all $v\in \mathscr V_a(\B_1^+, -1, 0)$ such that $v=0$ on $(\partial \B_1)^+\times(-1,0)$.  Here
\begin{equation*}
\Psi_\ve(v)=\int_{B_1}\zeta_\ve(v(x))\,dx\quad\text{with}\quad
\zeta_\ve(s)=
\begin{cases}
\ve s+\frac{s^2}{2\ve}+\ve^3,&s\leq -2\ve^2,\\
\frac{s^2}{4\ve},&-2\ve^2<s<0,\\
0,\qquad&s\geq 0.
\end{cases}
\end{equation*}
We explicitly observe that the  characterization of \eqref{pen1} in terms of  the variational inequality \eqref{varp} crucially uses the convexity of $\zeta_\ve$. For the  existence of solutions to the penalized problem, we refer to Section~5.6 in \cite{DL}. 

It also  follows immediately from the definitions that, for any $v\in \mathscr V_a(\B_1^+,-1,0)$ and any subsequence $\ve_j\to 0$ as $j\to\infty$, one has
\begin{equation}\label{cl1}
\int_{-1}^0 \Psi_{\ve_j}(v(t))dt\to \int_{-1}^0 \Psi(v(t))dt, \quad j\to\infty.
\end{equation}
We now proceed to show that if $v_{\ve_j}\to v$ and $\partial_t v_{\ve_j}\to \partial_t v$ weakly in $\mathscr V_a(\B_1^+,-1,0)$ as $j\to\infty$, and $\int_{-1}^0 \Psi_{\ve_j}(v_{\ve_j})dt\leq C$ (for some positive constant $C$ independent of $j$), then
\begin{equation}\label{claim1}
\liminf_{j\to\infty} \int_{-1}^0 \Psi_{\ve_j}(v_{\ve_j})dt\geq \int_{-1}^0 \Psi(v)dt.
\end{equation}
We begin by observing that, thanks to \cite[Theorem~2.8]{Ne}, the weak convergence of $v_{\ve_j}$ and $\partial_t v_{\ve_j}$ yields the strong convergence of $v_{\ve_j}$to $v$ in  $L^2(Q_1 )$. We define
\begin{equation*}
M(s)=\begin{cases}
\frac{s^2}{4},&s<0,\\
0,&s\geq 0.
\end{cases}
\end{equation*}
We then have
\begin{equation*}
C\geq \int_{-1}^0 \Psi_{\ve_j}(v_{\ve_j})dt=\int_{-1}^0\int_{B_1}\zeta_{\ve_j}(v_{\ve_j})\,dxdt\geq \frac{1}{\ve_j}\int_{-1}^0\int_{B_1}M(v_{\ve_j})\,dxdt,
\end{equation*}
from which we infer
\begin{equation*}
\int_{-1}^0\int_{B_1}M(v_{\ve_j})\,dxdt\to 0, \quad j\to\infty.
\end{equation*}
On the other hand,  it follows from the strong convergence in $L^2(Q_1)$ that
\begin{equation*}
\int_{-1}^0\int_{B_1}M(v_{\ve_j})\,dxdt\to \int_{-1}^0\int_{B_1}M(v)\,dxdt, \quad j\to\infty,
\end{equation*}
which in turn gives
\begin{equation*}
\int_{-1}^0\int_{B_1}M(v)\,dxdt=0.
\end{equation*}
Hence, $M(v)=0$ and  consequently $v\geq 0$ a.e.\ in $B_1$. This yields the validity of \eqref{claim1}. We have thus shown that the assumptions in Theorems 5.1 and 5.2 in \cite{DL} are satisfied. Hence,  we can invoke such theorems to obtain the following existence and uniqueness result.
\begin{thrm}\label{Ext}
Given $F\in L^\infty(\mathbb{Q}_1^+)$, with $\partial_t F\in L^\infty(\mathbb{Q}_1^+)$, and  $\tilde \phi_0 \in W_\infty^{2,1}(\mathbb{B}_1^+)$, there exists a unique function
$$
V\in \mathscr V_a(\B_1^+,0,1),\quad\text{with}\quad\partial_t V\in \mathscr V_a(\B_1^+,0,1)\cap L^\infty(0,1;L^2(\mathbb{B}_1^+, |y|^a dX)),
$$
which is weak solution of \eqref{epb333}. In addition, if $V_{\ve_j}$
is a solution of \eqref{pen1} for $\epsilon_j\to 0$, then $V_{\ve_j}\to V$ weakly in $\mathscr V_a(\Q_1^+,0,1)$, and $\partial_t V_{\ve_j}\to \partial_t V$ weakly in $\mathscr V_a(\Q_1^+,0,1)$ and star-weakly in $L^\infty(0,1;L^2(\mathbb{B}_1^+, |y|^a dX))$.
\end{thrm}

\begin{rmrk}
In view of our discussion above,  the corresponding existence and uniqueness result for the original thin obstacle problem as in \eqref{epb22} also follows.
\end{rmrk}

\section{Reduction to an elliptic thin obstacle problem and localized estimates}\label{lreg}

In this section we assume that the function $U$ be a solution in $\Q_1^+$ to the variational problem \eqref{epb333} with zero obstacle but with possibly nonzero lateral boundary conditions. Throughout, we will indicate with 
\[
\Lambda(U) = \{(x,t)\in Q_1\mid U(x,t,0) = 0\},
\]
the coincidence set of $U$, and with $\Gamma(U) = \p_{Q_1}\Lambda(U)$ its free boundary. On the right-hand side we assume that $F\in L^\infty(\mathbb{Q}_1^+)$ and $\partial_t F\in L^\infty(\mathbb{Q}_1^+)$. We  first establish optimal regularity estimates  by reduction to an elliptic thin obstacle problem. Subsequently, we prove localized  regularity estimates for the derivatives of $U$ independent of the boundary conditions. Such local estimates are critical in the blowup analysis in Section~6 in \cite{BDGP2}, where the structure of the singular set is studied. The reader should be aware that we will often pass from a problem in $\Q_1^+$ to one in $\Q_1$, while keeping the same notation for the data of the problem. Whenever we do so, we are thinking of having extended the relevant functions to the whole of $\Q_1$ by even reflection in $y$. The same applies when we consider a time-independent problem in $\B_1^+$ and pass to one in $\B_1$.

\subsection{Optimal regularity  estimate}\label{SS:optreg}

In this subsection we establish an optimal regularity estimate  for  $U(\cdot, y, \cdot)$ when considered as a function of $(x,t)$. Such a result is analogous to that in Corollary 6.10 in \cite{CSS}.

We start by establishing the local boundedness of $U_t$. We mainly
follow the approach \cite{PZ}. For a small $h>0$, consider the quantities
\begin{align*}
U^h(X,t)&=\frac{U(X,t)-U(X,t-h)}{h},\\
F^h(X,t)&=\frac{F(X,t)-F(X,t-h)}{h}.
\end{align*}

\medskip
\noindent
\emph{Claim:} The positive and negative parts of
  $U^h$ satisfy
\begin{equation}\label{re}
\La((U^h)^\pm)\leq |y|^a(F^h)^\pm\quad\text{in }\mathbb{Q}_{3/4}.
\end{equation}

We use the weak formulation of the thin obstacle problem in terms of variational inequalities. Thus, if
\[
\mathscr K_U=\{v\in \mathscr V_a(\B_1^+,-1,0)\mid v\geq 0\text{ on }Q_{3/4}, \ v=U \text{ on } (\partial_p \Q_{3/4})^+\},
\]
then $U\in\mathscr K_U$, $U_t\in L^2(\Q_{3/4}^+, |y|^a dXdt)$ and for a.e.\ $t\in(-(3/4)^2,0]$
\begin{equation}\label{epb9999}
\begin{multlined}
\int_{\mathbb{B}_{3/4}^+}
\langle \nabla_X V,\nabla_X(v -U)\rangle
|y|^adX+\int_{\mathbb{B}_{3/4}^+} U_t(v -U)\,|y|^a dX\\\geq\int_{\mathbb{B}_{3/4}^+} F(v-U)\ |y|^adX,
\end{multlined}
\end{equation}
for any $v\in \mathscr K_U$. 
To proceed, let $\chi\in C^\infty(\R)$ be such that
$$
\chi'\geq 0\text{ on }\R,\quad \chi=0\text{ on }(-\infty,1],\quad
\chi=1\text{ on }[2,\infty).
$$
Then for a nonnegative $\eta\in C_0^\infty(\Q_{3/4})$ and $\epsilon>0$ we let
$$
\eta_\epsilon=\eta\,\chi (U^h/\epsilon).
$$
We next note that if $\tau>0$ is a small number such that $\tau \eta
<\epsilon$ in $\Q_{3/4}$, then $v=U\pm \tau \eta_\epsilon\in \mathscr K_U$
and hence from \eqref{epb9999} we will have for a.e.\ $t\in(-(3/4)^2,0]$
$$
\int_{\B_{3/4}^+}\langle \nabla_X U,\nabla_X \eta_\epsilon\rangle |y|^a dX
+ \int_{\B_{3/4}^+} U_t\,\eta_\epsilon\,|y|^a dX=\int_{\B_{3/4}^+}
F\eta_\epsilon\,|y|^adX.
$$
On the other hand, writing the variational inequality similar to \eqref{epb9999} for the time shift $U(\cdot,\cdot-h)$ and taking $v(X,t)=U(X,t-h)+\eta_\epsilon(X,t)\in \mathscr
K_{U(\cdot,\cdot-h)}$, we have for a.e.\ $t\in(-(3/4)^2,0]$
$$
\begin{multlined}\int_{\B_{3/4}^+}\langle \nabla_X U(\cdot,t-h),\nabla_X \eta_\epsilon\rangle |y|^a dX
+ \int_{\B_{3/4}^+} U_t(\cdot,t-h)\,\eta_\epsilon\,|y|^a dX\\\geq\int_{\B_{3/4}^+}
F(\cdot,t-h)\eta_\epsilon\,|y|^a dX
\end{multlined}
$$
and hence, taking the difference, we obtain
$$
\int_{\B_{3/4}^+}\langle \nabla_X U^h,\nabla_X \eta_\epsilon\rangle |y|^a dX
+ \int_{\B_{3/4}^+} U_t^h\,\eta_\epsilon\,|y|^a dX\leq \int_{\B_{3/4}^+}
F^h\eta_\epsilon\,|y|^adX.
$$
Now, noticing that 
$$
\nabla_X\eta_\epsilon=\nabla_X\eta\, \chi(U^h/\epsilon)+\frac{1}{\epsilon}\eta\,
\chi'(U^h/\epsilon)\nabla_X U^h, 
$$
we can infer
$$
\begin{multlined}
\int_{\B_{3/4}^+}\langle \nabla_X U^h,\nabla_X \eta\rangle
\chi(U^h/\epsilon)|y|^a dX
+ \int_{\B_{3/4}^+} U_t^h\,\eta\,\chi(U^h/\epsilon)|y|^a dX\\ \leq \int_{\B_{3/4}^+}
(F^h)^+\eta\,\chi(U^h/\epsilon)|y|^a dX
\end{multlined}
$$
and passing to the limit as $\epsilon\to 0$, with the help of
dominated convergence theorem, we obtain that for a.e.\ $t\in(-(3/4)^2,0]$
$$
\int_{\B_{3/4}^+}\langle \nabla_X (U^h)^+,\nabla_X \eta\rangle|y|^a dX
+ \int_{\B_{3/4}^+} (U^h)^+_t\,\eta\,|y|^a dX\leq \int_{\B_{3/4}^+}
(F^h)^+\eta\,|y|^adX.
$$
As the nonnegative test function $\eta\in C^\infty(\Q_{3/4})$ was
arbitrary, the above inequality implies that
$$
\mathscr L_a ((U^h)^+)\leq |y|^a(F^h)^+\quad\text{in }\Q_{3/4}.
$$
Using a similar argument, we also obtain that
$$
\mathscr L_a ((U^h)^-)\leq |y|^a(F^h)^-\quad\text{in }\Q_{3/4}.
$$
This establishes \eqref{re}.

Once we have \eqref{re}, by first applying the subsolution estimate in \cite{CSe}, and then letting $h \to 0$, we infer for some constant $C=C(n,a)>0$,
\begin{equation}\label{p1}
\|U_t\|_{L^{\infty}(\Q_{1/2}^+)} \leq C (\|U_t\|_{L^2(\Q_{3/4}^+,|y|^a dXdt)} + \|F_t\|_{L^{\infty}(\Q_{3/4}^+)}).
\end{equation}
This proves the local boundedness of the time derivative of $U$. 

With this information in hand, if we let $\tilde F= U_t + F$, we infer that at each fixed time $t$, $U(\cdot,t)$ solves the following elliptic thin obstacle problem 
\begin{equation}\label{ell}
\begin{cases}
L_a U \overset{\rm def}{=} \operatorname{div}\ (y^a \nabla U)= y^a \tilde
F & \text{in}\ \B_1^+,
\\
\min\{U(x,0,t),-\partial_y^a U(x,0,t)\}=0 &\text{on }B_1.
\end{cases}
\end{equation}
The reader should bear in mind that for every $t>0$ one has $\tilde F(\cdot,-y,t) = \tilde F(\cdot,y,t)$. 

\medskip

We now proceed with the proof of the optimal regularity estimate. Since for a given $t_0$ the function $\tilde F(\cdot, t_0)$ is bounded, we are precisely in the improved situation considered in Theorem~6.1 and 6.2 in \cite{CDS} (note that \cite{CSS} instead requires a Lipschitz right-hand side) and hence from the results there, we can infer that  $U(\cdot, 0, t_0)$ is in $C^{1,\frac{1-a}2}$ at every time level $t_0$. From this and an argument using cut-offs  as in the proof of Lemma~4.1 in \cite{CSS}, we conclude that $U(\cdot, y, t_0)$ is in $C^{1,\frac{1-a}2}$ for $y<1/2$. Coupled with the boundedness of $U_t$, this allows to conclude that $U(\cdot, y,\cdot) \in H^{\frac{3-a}2, \frac{3-a}4}$ for $y <1/2$. To see this, we note that 
\begin{align*}
& |U(x,\cdot,t) - U(0,\cdot,0) - \langle\nabla_x
                 U(0,\cdot,0),x\rangle|\\
&\qquad \leq |U(x,\cdot,0) - U(0,\cdot,0) - \langle\nabla_x U(0,\cdot,0),x\rangle|+ |U(x,\cdot,t) - U(x,\cdot,0)| \\
&\qquad\leq C|x|^{\frac{3-a}2} + C |t| \leq C(|x| + |t|^{1/2})^{\frac{3-a}2}.
\end{align*}
From this $\frac{3-a}2$-order of approximation at every point, a standard argument shows $
\nabla_x U \in H^{s, s/2}$. This proves the optimal regularity estimate. We summarize all of this in the following. 

\begin{thrm}\label{opt}
Let $U$ be a solution to 
\[
\begin{cases}
\La U = |y|^a F &\text{in }\Q_1^{+},\\
\min\{U(x,0,t),-\partial_y^a U(x,0,t)\}=0 &\text{on }Q_1.
\end{cases}
\]  Then for every $y\in[0,1/2]$ we have $U(\cdot,y, \cdot) \in H_{\rm loc}^{\frac{3-a}2, \frac{3-a}4}$.
\end{thrm}  

We emphasize that the elliptic regularity in Theorem~\ref{opt} is optimal because of the following prototypical function
\begin{equation}\label{hatv0}
\hat v_0(X) = \hat v_0(x,y)= c \left(x_n+\sqrt{x_n^2+y^2}\right)^{\frac{1-a}2}\left( x_n-\frac{1-a}2 \sqrt{x_n^2+y^2}\right),
\end{equation}
see \cite{GPPS}. Such $\hat v_0$ is a global solution in $\B_1$ of the problem \eqref{ell} with  $\tilde F\equiv 0$ (this corresponds to a problem with zero obstacle). Note that we have $\hat v_0(x,-y) = \hat v_0(x,y)$, $\hat v_0(x,0) \ge 0$ in $B_1$, and that $\hat v_0$ is homogeneous of degree 
\begin{equation}\label{kappazero}
\kappa_0 = \frac{3-a}2. 
\end{equation}

\subsection{Localized regularity estimates}
In this subsection we obtain localized H\"older estimates in $(X,t)$ for $\nabla_x U$ and $y^a U_y$, up to the thin manifold $\{y=0\}$. We assume that $\nabla_x F$ is bounded and that $F_y= O(y)$ in $\Q_1^+$.  More precisely, we suppose that for some $K>0$ the following bounds hold
\begin{equation}\label{der1}
\|F\|_{L^{\infty}(\Q_1^+)} \le K,\quad \|\nabla_x F\|_{L^{\infty}(\Q_1^+)}  + \|F_t\|_{L^{\infty}(\Q_1^+)}  \leq K,\quad |\partial_y F|  \leq Ky.
\end{equation}
We note that \eqref{der1} is satisfied by  the functions $F_k$ in \cite[Section~3]{BDGP2}. Therefore, our regularity estimates in Theorem~\ref{localized} below can be applied  to situations such as those in \cite[Sections~3 and 6]{BDGP2}. 

\medskip

We proceed as follows.  We first show that $y^a U_y$ is continuous in
$(X,t)$ up to the thin set $\{y=0\}$. Again, from \cite{CDS}, it follows that, at every time level $t$, $y^a U_y$ is H\"older
continuous in $X$ up to $\{y=0\}$. For a more
self-contained proof of this fact, we refer to Theorem~\ref{schauder}
below. Now, recall the $C^{1+\alpha}_a$ norm defined in \eqref{wnorm}. Given $t_0$, let $\{t_j\}$  be a sequence of times converging
to $t_0$. Since for every $j\in \mathbb N$, $U(\cdot,t_j)$ solves an
elliptic thin obstacle problem with uniformly bounded right hand
side, from the elliptic regularity results in
\cite{CDS} we infer that $U(\cdot, t_j)$'s are uniformly bounded
in $C^{1+\alpha}_{a}(\overline{\B_{1/2}^+})$ for some $\alpha>0$. By Ascoli-Arzel\`a we infer that, up to a subsequence, $U(\cdot, t_j) \to
U_0$ in $C^{1+\beta}_{a}(\overline{\B_{1/2}})$ for all $\beta <
\alpha$. Also, away from $\{y=0\}$, since $U$ is a solution to a
uniformly parabolic PDE with bounded right-hand side, by
the De Giorgi-Nash-Moser theory we have that $U(\cdot,y,t_j) \to U(\cdot,y,
t_0)$ in $\{y >0\}$ pointwise, and this allows us to conclude that $U_0
\equiv U(\cdot,y,t_0)$. By the uniqueness of the limit, we can assert
that the whole sequence $U(\cdot,y,t_j) \to U(\cdot,y,t_0)$ in
$C^{1+\beta}_a(\overline{\B_{1/2}^+})$  for all $\beta < \alpha$ and
this in particular implies the continuity of $y^a U_y$ in the variable
$t$ up to $\{y=0\}$.

Having established the continuity of $y^a U_y$, similarly to Section~4 in \cite{DGPT}, we now  define the \emph{extended free boundary} as follows.
\begin{equation}\label{efb}
\Gamma_*(U)= \p_{Q_1} \{ (x,t)\in Q_1 \mid U(x, 0, t)=0, \partial_y^a U(x,0,t)=0\}.
\end{equation}
If $(x_0,t_0) \in \Gamma_*(U)$, then  thanks to the continuity of
$\partial_y^a U$, $\nabla_x U$ and  $U$ on $\{y=0\}$,  we have that at $x_0$ the following facts hold:
\begin{equation}\label{rel}
U(x_0, 0, t_0)=0,\quad \nabla_x U(x_0, 0,t_0)=0, \quad \partial_y^a U (x_0, 0, t_0) =0.
\end{equation}
Keeping in mind both the fact that $U(\cdot,t_0)$ solves the elliptic thin obstacle problem with bounded right hand side,  and \eqref{rel},  from \cite{CDS} it follows
\begin{equation}\label{g1}
\|U(\cdot, t_0)\|_{L^{\infty}(\B_r^+(x_0))} \leq C r^{\frac{3-a}2}.
\end{equation}
In turn, \eqref{g1} coupled with  the boundedness of $U_t$ yields
\begin{equation}\label{growth}
\|U\|_{L^{\infty}(\Q_r^+(x_0, t_0))} \leq C r^{\frac{3-a}2}.
\end{equation}
Next, for  $(X,t) \in \Q_{1/2}^+$, let $d(X,t)$ be the parabolic distance from the extended free boundary $\Gamma_*(U)$. As in the case $a=0$ analyzed in \cite{DGPT}, from the estimate \eqref{growth} it follows in a straightforward way that 
\begin{equation}\label{k1}
|U(X,t)| \leq  C d(X,t)^{\frac{3-a}2}.
\end{equation}
We now consider the intersection $\Q_{d}(X,t) \cap Q_1$, where $d=d(X,t)$. Since there are no points of $\Gamma_{*}(U)$ in this set, we have two possibilities. Either (i) $U>0$ on $\Q_d(X,t) \cap Q_1$; or,
(ii) $U \equiv 0$ on $\Q_d(X,t) \cap Q_1$. 
If (i) occurs, then we have $\partial_y^aU =0$ on $\Q_d(X,t) \cap Q_1$. Thus, we can even reflect across $\{y=0\}$. By scaling the estimate in Proposition~\ref{ac1} and taking Remark \ref{conh} into account, in view of \eqref{k1} we obtain  
\begin{equation}\label{k2}
\begin{aligned}
|\nabla_x U(X, t) | &\leq Cd^{\frac{1-a}2},
\\
|y^a U_y(X, t)| &\leq Cd^{\frac{1+a}2}.
\end{aligned}
\end{equation}
If instead (ii) occurs, then \eqref{k2} follows from the scaled version of the estimate in Lemma~\ref{odd}. 

We now take points $(X^i, t^i) \in \Q_{1/2}^+$, $i=1,2$ and let $d_{i}= d(X^i, t^i)$. We also set $\delta= |(X^1 - X^{2}, t^{1}-t^{2})|$. Without loss of generality, we may assume that $d_1\geq d_2$. There exist two possibilities: (a) $\delta \ge \frac{1}{2} d_1$; or, (b) $\delta < \frac{1}{2} d_1$.
If (a) occurs, it follows from \eqref{k2} that
\[
\begin{aligned}
|\nabla_x U(X^1, t^1) - \nabla_x U(X^2, t^2) | & \leq |\nabla_x U(X^1, t^1)| + |\nabla_x U(X^2, t^2)| \\
 &\leq C (d_1^{\frac{1-a}2} + d_2^{\frac{1-a}2}) \leq C \delta^{{\frac{1-a}2}},
\\
| y^a U_y(X^1, t^1) - y^a U_y(X^2, t^2) | &\leq | y^a U_y(X^1, t^1)| + | y^a U_y(X^2, t^2)| \\
& \leq C (d_1^{\frac{1+a}2} + d_2^{\frac{1+a}2}) \leq C \delta^{\frac{1+a}2}.
\end{aligned}
\]
If instead (b) occurs, then both $(X^{i},t^{i}) \in \Q_{d_1/2}(X^1,t^1)$, and  we have from \eqref{k1}
\begin{equation}\label{tr}
\|U\|_{L^{\infty}(\Q_{d_1/2} ( X^1, t^1))} \leq Cd_1^{\frac{3-a}2}.
\end{equation}
From the scaled version of the estimate in Proposition~\ref{ac1}, or from Lemma~\ref{odd} (depending on whether  $U>0$ in $\Q_{d_1/2} \cap Q_1$ or not), it follows that for $\beta = \min\{\frac{1-a}2,\alpha\}$, with $\alpha$ as in \eqref{ert} in Lemma~\ref{odd}, the following holds 
\[
|\nabla_x U(X^1, t^{1}) - \nabla_x U(X^2, t^2)| \leq 
 \frac{C}{d_1^{1+\beta}}\left(\|U\|_{L^{\infty}(\Q_{d_1/2} ( X^1, t^1))} + d_1^2 K\right) \delta^\beta
\leq C \delta^\beta,
\]
where in the last inequality, we have also used \eqref{tr} and the fact that $\beta \leq  \frac{1-a}2$.  Likewise, for $\gamma=\min\{\frac{1+a}2, \alpha\}$,  we find 
\[
| y^a U_y(X^1, t^1) - y^a U_y(X^2, t^2) | \leq C\delta^{\gamma}.
\]
We can thus finally assert that $U_t \in L^{\infty}_{\rm loc}, \nabla_x U \in H^{\alpha_0, \frac{\alpha_0}{2}}_{\rm loc}$ and $y^a U_y \in H^{\alpha_0, \frac{\alpha_0}{2}}_{\rm loc}$ up to the thin manifold $\{y=0\}$, for some $\alpha_0>0$. Using such H\"older regularity of $\nabla_x U, y^aU_y$, and the boundedness of $U_t$,
we can at this point argue as in the proof of Lemma~5.1 in \cite{BDGP2}, and conclude that the following $W^{2,2}$ type estimate holds for $\rho<1$,
\begin{equation}\label{step1}
\int_{\Q_{\rho}^+} (  |\nabla U|^2  + |\nabla U_{x_i}|^2 + U_t^2 )  |y|^a \leq C(n,\rho) \int_{\Q_1^+} (U^2  +  F^2)  |y|^a,
\end{equation}
Taking the estimate \eqref{p1} into account, from the above discussion and from \eqref{step1}, we finally obtain the following localized regularity estimates. 

\begin{thrm}\label{localized}
Suppose that $F$ satisfy the bounds in \eqref{der1} for some $K>0$. Let $U$ be a solution to 
\[
\begin{cases}
\La U = |y|^a F &\text{in }\Q_1^{+},\\
\min\{U(x,0,t),-\partial_y^a U(x,0,t)\}=0 &\text{on }Q_1.
\end{cases}
\] 
Then the growth estimate as in \eqref{growth} holds near any free boundary point $(x_0, t_0) \in \Gamma_*(U)$. Moreover,  there exists $\alpha_0\in (0,1)$ such that $U$ satisfy the following local estimate 
\[
\|y^a U_y\|_{H^{\alpha_0,
      \frac{\alpha_0}{2}}(\overline{\Q_{1/2}^+})}  + \|\nabla_x
  U\|_{H^{\alpha_0, \frac{\alpha_0}{2}}(\overline{\Q_{1/2}^+})} +
  \|U_t\|_{L^{\infty}(\Q_{1/2}^+)}\leq C
\left(\|U\|_{L^2(\Q_1^+,|y|^a dXdt)} +K\right).
\]
\end{thrm}

\section{Regular free boundary points}\label{S:rfb}

In this section we analyze the so-called regular free boundary points. We begin with the thin obstacle problem \eqref{epb}, where $\psi \in H^{\ell, \ell/2}$ for some $\ell \geq 4$. In view of the reductions in \cite[Section~3]{BDGP2}, \eqref{epb} is in turn equivalent to analyzing the following global problem with zero obstacle,
\begin{equation}\label{e0}
 \begin{cases}
\La U = |y|^a F&\text{in }\Sa_1^+,
\\
\min\{U(x,0,t), -\partial_y^a U(x,0,t)\}=0 &\text{on }S_1,
\end{cases}
\end{equation}
where  $\Sa_1^+= \Rnp \times (-1, 0]$, $S_1 = \Rn\times (-1,0]$, and  $F$ satisfies
\begin{align}\label{fbound}
|F(X,t)| \leq  M |(X,t)|^{\ell -2} &\quad\text{for }(X,t)\in\Sa_1^+,\\
\label{nablafbound}
 |\nabla_X F(X, t)| \leq M |(X,t)|^{\ell -3}&\quad\text{for }(X,t)\in\Q_{1/2}^+,\\
\label{ptfbound}
 |\p_t F(X, t)| \leq M |(X,t)|^{\ell -4}&\quad\text{for }(X,t)\in\Q_{1/2}^+.
\end{align}

We now fix an extended free boundary point of $U$ in \eqref{e0} and, without loss of generality, we assume that it be the origin, thus $(0,0)\in \Gamma_{*}(U)$. We next consider the quantity
\begin{align}\label{HU}
H^{\rm par}(U,r) & = \frac{1}{r^2} \int_{\Sa_r^+} U^2\ \Gb |y|^{a} dX dt,
\end{align}
where
\[
\Gb(X,t)= \frac{(4 \pi)^{-n/2}}{2^a \Gamma(\frac{a+1}{2})} |t|^{-\frac{n+a+1}{2}} e^{\frac{|X|^2}{4t}}.
\]
The following result is Theorem~4.8 in \cite{BDGP2}.

\begin{thrm}[Monotonicity formula of Almgren-Poon type] \label{T:poon}
Let $U$ solve \eqref{e0}  with $F$ satisfying \eqref{fbound}. Then, for every $\sigma\in (0,1)$ there exists a constant $C>0$, depending on $n, a, M$ and $\sigma$,  such that the function
\begin{equation}\label{parfreq}
r\mapsto \Phi_{\ell,\sigma}^{\rm par}(U,r) \overset{\rm def}{=} \frac{1}{2} r e^{C r^{1-\sigma}} \frac{d}{dr}\log \max\left\{H^{\rm par}(U,r),\ r^{2\ell - 2+2\sigma}\right\} + 4(e^{Cr^{1-\sigma}}-1),
\end{equation}
is monotone nondecreasing on $(0,1)$.
In particular, the following limit exists
\begin{equation}\label{ka}
\kappa_{U} (0,0)= \Phi_{\ell,\sigma}^{\rm par}(U,0^+) \overset{\rm def}{=} \lim_{r\to 0^+}\Phi_{\ell,\sigma}^{\rm par}(U,r).
\end{equation}
\end{thrm}

\begin{dfn}
The  limit $\kappa_{U}(0,0)$ is called the \emph{parabolic frequency} at $(0,0) \in \Gamma_{*}(U)$. By translation, we can likewise define $\kappa_U(x_0, t_0)$ at every other free boundary point $(x_0, t_0)$ of $U$. 
\end{dfn}

We recall the definition \eqref{kappazero} of $\kappa_0$. The following gap result states, in particular, that $\kappa_0$ is the lowest possible frequency, see \cite[Lemma~7.2]{BDGP2}. 

\begin{lemma}\label{gap1}
Let   $\kappa_U=\kappa_U(0,0), \ell, \sigma$ be as in Theorem~\ref{T:poon}, and suppose that $\kappa_U \leq \ell - 1 + \sigma$. Then,  either  $\kappa_U = \kappa_0$, or $\kappa \geq 2$.
\end{lemma}

We can now introduce the notion of regular free boundary points. 

\begin{dfn}\label{reg}
We say that $(x_0,t_0)\in \Gamma_*(U)$ is a \emph{regular free boundary point} if the parabolic frequency $\kappa_U(x_0,t_0)=\kappa_0$. We denote this set of such points by $\Gamma_{\kappa_0}(U)$ and call it the \emph{regular set} of $U$.
\end{dfn}

\subsection{Results from the elliptic theory}\label{S:ellipticalmgren}

 We next recall some results from the elliptic theory that will be needed in our subsequent analysis of the parabolic problem. As before, let $a\in (-1,1)$. Given a function $v\in W^{1,2}_{\rm loc}(\B_1,|y|^a dX)\cap C(\B_1)$, for $0<r<1$ we introduce the quantities
\begin{align}\label{Hv}
H(r)=H(v,r) &\overset{\rm def}{=} \int_{\p \B_r} v^2 |y|^{a} d\sigma,\\
\intertext{where $\sigma$ indicates the standard $n$-dimensional surface measure on $\partial \BB$,}
\label{Gv}
G(r)=G(v,r) &\overset{\rm def}{=} \int_{\B_r} v^2 |y|^{a} dX,\\
\intertext{and} 
\label{D}
D(r)=D(v,r) &\overset{\rm def}{=} \int_{\BB} |\nabla v|^{2} |y|^{a} dX.\\
\intertext{We also consider the total energy of $v$}
\label{I}
I(r) = I(v,r) &\overset{\rm def}{=} \int_{\p \BB} v \langle\nabla v,\nu\rangle |y|^a,
\end{align}
where $\nu$ indicates the outer unit normal to a domain in $\Rnn$. 
The \emph{frequency} of $v$ is defined as
\begin{equation}\label{freq}
N(r)=N(v,r) \overset{\rm def}{=} \frac{rI(r)}{H(r)}.
\end{equation}

To simplify the notation, in the sequel we omit writing the measures $d\sigma$ an $dX$ in all the surface and volume integrals involved.

\subsection{An improved monotonicity formula}

The elliptic Almgren type monotonicity formula in \cite[Theorem~3.1]{CSS} is valid under the assumption that the right-hand side be in $C^{0,1}(\B_1)$. In our reduction of the parabolic problem \eqref{epb22} to an elliptic one, it is essential that we deal with a right-hand side which is only in $L^\infty(\B_1)$. This comes from the fact that our right-hand side contains $U_t$, which by Theorem~\ref{localized} is only bounded. We thus need an improvement of the above cited result in \cite{CSS}, similar to that first proved in \cite[Theorem~1.4]{GG} for case $a=0$. 

We now consider a solution $v\in W^{1,2}_{\rm loc}(\B_1,|y|^a dX)\cap C(\B_1)$ to the following elliptic thin obstacle problem with zero obstacle,
\begin{equation}\label{epb222}
 \begin{cases}
L_a v = |y|^a f &\text{in}\ \B_1^+ \cup \B_1^-,
\\
\min\{v(x,0),-\partial_y^a v(x,0)\}=0 &\text{on }B_1,\\
v(x,-y) = v(x,y) &\text{in }\B_1.
\end{cases}
\end{equation}
We remark explicitly that, if we let $\Lambda = \Lambda(v) = \{(x,0)\in \B_1\mid v(x,0) = 0\}$,
and denote by $\mathscr H^n$ the $n$-dimensional Hausdorff measure in $\Rnn$, then a solution $v$ to \eqref{epb222} satisfies the equation
\[
L_a v= |y|^a f\ +\ 2 \pa\ \mathscr H^n\lfloor_{\Lambda}\quad\text{in}\ \mathscr D'(\B_1).
\]
This means that for every $\vf\in W^{1,2}_0(\B_1,|y|^a dX)$, we have
\begin{equation}\label{epb222w}
\int_{\B_1} \langle\nabla v,\nabla \vf\rangle |y|^a = - \int_{\B_1} \vf f |y|^a - 2 \int_{B_1\cap \Lambda} \vf \pax,
\end{equation}
where in the last integral the function $\vf$ must be interpreted in the sense of traces. We also  define the free boundary in the following way
\[
\Gamma(v)= \partial_{B_1} \Lambda(v)
\]
We now state a few results whose proofs  in this generality can be found, for instance, in \cite{CDS} and \cite{GPPS}.  

\begin{lemma}[Caccioppoli inequality]\label{L:caccio}
Let $v\in W^{1,2}(\B_1,|y|^a dX)\cap C(\B_1)$ be a solution to \eqref{epb222}. Then, there exists a constant $C>0$ depending on $n, a$ such that
\begin{equation}\label{caccio}
\int_{\BB} |\nabla v|^2 |y|^a \le \frac{C}{r^2} \int_{\B_{2r}} v^2 |y|^a + \|f\|_{L^\infty(\B_1)} \int_{\B_{2r}} v^2 |y|^a.
\end{equation}
\end{lemma}

In our subsequent analysis we also need the following Schauder type estimate. 

\begin{thrm}\label{schauder}
Let $v$ be a solution to \eqref{epb222} with $f$ bounded. Then, there exists $\alpha>0$ such that
\[
\|v\|_{C^{ \alpha}(\overline{\B_{1/2}^+})} + \|\nabla_x v\|_{C^{ \alpha}(\overline{\B_{1/2}^+})} + \|y^a v_y\|_{C^{\alpha}(\overline{\B_{1/2}^+})} \leq C,
\]
for some universal $C$ depending on $\|U\|_{L^{2}(\B_1^+, y^a dX)}$ and $\|f\|_{L^\infty(\B_1)}$.
\end{thrm}

\begin{proof}
This essentially follows from the regularity results in \cite{CDS}, but we nevertheless provide details for the sake of completeness, since such a result is not explicitly stated  there.   We first evenly reflect $f$ and then  let $w$ be the solution to 
\[
\begin{cases}
L_a w = |y|^a f &\text{in}\ \B_1,
\\
w = 0 &\text{on}\ \p \B_1.
\end{cases}
\]
Recall that, by uniqueness, we have $w(x,-y) = w(x,y)$. Therefore, in view of Proposition~\ref{ac1},    for every $\gamma \in (0,1)$ we have $w\in C^{1,\gamma}_{\rm loc}(\B_1)$.  By symmetry, it follows for  $x\in B_1$
\[
\partial_y^a w(x,0)=0.
\]
Thus, if we  define $u = v- w$ and $\psi(x) {=} - w(x,0)$,  it is clear that   the function $u$ solves the problem
\begin{equation}\label{b1}
 \begin{cases}
L_a u = 0 &\text{in}\ \B_1^+ \cup \B_1^-,
\\
\min\{u(x,0)-\psi(x),-\partial_y^a u(x,0)\}=0 &\text{on }B_1,\\
u(x,-y) = u(x,y) &\text{in }\B_1.
\end{cases}
\end{equation}
Since  $\psi\in C^{1,\gamma}_{\rm loc}(B_1)$ for every $\gamma \in (0,1)$, an application of Theorem~6.1 in \cite{CDS} gives $u(x, 0) \in C^{1,s}(B_1)$. Then, using cut-offs and an argument as in the proof of Lemma~4.1 in \cite{CSS},  we can assert that  $\lim_{y \to 0} y^a u_y \in C^{\beta}(B_1)$, for some $\beta>0$. Adapting the difference quotient argument in the proof of \cite[Lemma~2.17]{PP}, we can finally conclude  that, for some  $\delta>0$, $\nabla_x u, y^a u_y \in C^{\delta}(\overline{B_{1/2}^+})$. Finally, since $v= u+w$, the desired conclusion follows. 
\end{proof}

\begin{lemma}\label{L:H'}
One has for every $r\in (0,1)$
\begin{equation}\label{H'}
H'(r) = \frac{n+a}{r} H(r) + 2 I(r).
\end{equation}
\end{lemma}

Our next results connect $I(r)$ to $D(r)$, and give the first variation of both. 

\begin{lemma}\label{L:ID}
For every $r\in (0,1)$ one has
\begin{equation}\label{ID}
I(r) = D(r) + \int_{\BB} v f |y|^a.
\end{equation}
\end{lemma}

\begin{lemma}\label{L:fve}
For every $r\in (0,1)$ one has
\begin{align}\label{fve1}
D'(r) &= \frac{n+a-1}{r} D(r) + 2 \int_{\p \BB} \langle\nabla v,\nu\rangle^2  |y|^a - \frac 2r \int_{\BB} \langle\nabla v,X\rangle f\\
\label{fve2}
I'(r) & = \frac{n+a-1}{r} I(r) + 2 \int_{\p \BB} \langle\nabla v,\nu\rangle^2  |y|^a - \frac 2r \int_{\BB} \langle\nabla v,X\rangle f |y|^a
\\
&\qquad - \int_{\p \BB} vf |y|^a - \frac{n+a-1}{r} \int_{\BB} vf |y|^a.
\notag
\end{align}

\end{lemma}

Given  $\delta\in (0,1)$, we introduce the set
\begin{equation}\label{lo}
\Lambda_\delta  = \left\{r\in (0,1)\mid H(r)> r^{n+a+2 + 2\delta}\right\}.
\end{equation}

\begin{lemma}\label{L:HD}
Assume  $v(0)=0$. There exist $C, r_0>0$, depending on $n, a, \|f\|_{L^\infty(\B_1)}$ and  $\delta\in (0,1)$, such that for $r\in \Lambda_\delta\cap (0,r_0)$ one has
\begin{equation}\label{HD}
H(r) \le C r D(r).
\end{equation}
\end{lemma}

We also need the following result.

\begin{lemma}\label{L:DI}
Assume $v(0)=0$. There exists $r_0>0$, depending on $n, a, \|f\|_{L^\infty(\B_1)}$ and  $\delta\in (0,1)$,  such that if $r\in \Lambda_\delta\cap (0,r_0)$, then one has
\[
D(r) \le 2 I(r).
\]
\end{lemma}

At this point we can state the relevant monotonicity formula for the elliptic thin obstacle problem \eqref{epb222}, see \cite[Theorem~1.4]{GG} for the case $a = 0$. 

\begin{thrm}[Truncated Almgren type frequency formula]\label{T:ellalmgren}
Let $v$ solve the obstacle problem \eqref{epb222} and suppose that $0\in \Gamma(v)$. For any $\delta\in(0,1)$ there exist constants $C, r_0$, depending on $n, a, \|f\|_{L^\infty(\Rnn)}$ and $\delta$,  such that the function
\begin{equation}\label{generalized-frequency}
r\mapsto \Phi_\delta(v,r) \overset{\rm def}{=} r e^{C r^{1-\delta}} \frac{d}{dr}\log \max\left\{H(v,r),\ r^{n+a+2+2\delta}\right\},
\end{equation}
is monotone nondecreasing on $(0,r_0)$.
In particular, the following limit exists
\[
\kappa_v= \Phi_\delta(v,0^+) \overset{\rm def}{=} \lim_{r\to 0^+}\Phi_\delta(v,r).
\]
\end{thrm}

We now  define the family of nonhomogeneous  Almgren type rescalings. In the case $a = 0$ they were first introduced in \cite{ACS}.

\begin{dfn}\label{D:ar}
Let $v$ be a solution to \eqref{epb222} and assume that $0 \in \Gamma(v)$.   Consider the quantity
\[
d_r = \left(\frac{H(v,r)}{r^{n+a}}\right)^{1/2}.
\]
The \emph{Almgren rescalings} of $v$ at $X = 0$ are defined as follows
\begin{equation}\label{ar}
\tilde v_r(X) = \frac{v(rX)}{d_r},\quad X\in \B_{1/r}.
\end{equation}
\end{dfn}

The first obvious, yet important, observation is that
\begin{equation}\label{Har}
H(\tilde v_r,1) \equiv 1.
\end{equation}
Another basic property is the following: for every $0< r, \rho<1$ one has
\begin{equation}\label{Nar}
N(\tilde v_r,\rho) = N(v,r \rho).
\end{equation}

A crucial consequence of Theorem~\ref{T:ellalmgren} is the following compactness property of the Almgren rescalings.

\begin{thrm}\label{conv}
There exists a sequence $r_j \searrow 0$ such that for some $\gamma>0$ one has $\tilde v_{r_j} \to v_0$  in $C_a^{1+\gamma}(K)$ on compact subsets of $K$ of $\overline{\Rnp}$. Here, $v_0$ is a global solution of  the thin obstacle problem \eqref{epb222} with $f \equiv 0$. Also, $v_0$  is homogeneous of degree $(\kappa_v - n - a)/2$. Moreover, when $\kappa_v= n+ a + 2\kappa_0 = n+3$, we have that after a rotation of coordinates in $\Rn$, $v_0$ is of the form \eqref{hatv0}. 
\end{thrm}

\begin{proof}

First, similarly to \cite{CSS}, \cite{GG} and \cite{CDS}, from Theorem~\ref{T:ellalmgren}, the scaling properties of the frequency and from energy considerations, we infer that: 
\begin{enumerate}[1)]
\item $\tilde v_{r_j} \to v_0$  in $W^{1,2}_{\rm loc}( \overline{\Rnp}, |y|^a dX)$;   
\item $v_0$ is homogeneous of degree $(\kappa_v - n - a)/2$.
\end{enumerate}
  The convergence in $C_a^{1+\gamma}$ is then a consequence of the uniform Schauder estimates in Theorem~\ref{schauder} and the theorem of Ascoli-Arzel\`a.  
Finally, in the case when  $\kappa_v= n+a +2 \kappa_0$, the fact that $v_0$  takes the form \eqref{hatv0} follows from 
\cite[Proposition~5.5]{CSS}. 
\end{proof}

We also  have the following result on the frequency gap. We refer to  the discussion on page 926 in \cite{CDS} for a proof of this fact. Notice that although the functional in \cite{CDS} is a bit different from our $\Phi_\delta$, nevertheless both   have the same limit as $r \to 0$.

\begin{lemma}\label{L:gap}
Let $1+\delta> \kappa_0 = \frac{3-a}{2}$. Then, either $\Phi_\delta(v,0^+) = n + a + 2\kappa_0$, or 
$\Phi_\delta(v,0^+) \ge n + a + 2 + 2\delta$.
\end{lemma}

\begin{dfn}\label{D:reg}
Let $v$ be a solution of \eqref{epb222}. We say that $0\in \Gamma(v)$ is a \emph{regular free boundary point} if $
\Phi_\delta(v,0^+) = n + a + 2 \kappa_0$. Likewise,  we say that $X_0 = (x_0,0)$ is regular if such is the point $(0,0)$ for the function $v_{X_0}(X) = v(X+X_0)$. We denote by $\Gamma_{\kappa_0}(v)$ the set of all regular free
boundary points and we call it the \emph{regular set} of $v$.
\end{dfn}

We recall that if $U$ is a solution to \eqref{e0}, then $U(\cdot, t_0)$ solves the elliptic thin obstacle problem \eqref{epb222} in $\B_1^+$  with $f= U_t (\cdot, t_0)+ F(\cdot, t_0)$. We have the following lemma, which follows from \cite[Theorem~7.3]{BDGP2}. See also the discussion in   Remark 7.4 in the same paper.

\begin{lemma}\label{equiv}
Let $U$ be a solution to  \eqref{e0}. Then $(x_0, t_0)$ is a regular free boundary point for $U$ in the sense of Definition \ref{reg} if and only if $x_0$ is a regular free boundary point for $U(\cdot, t_0)$ in the sense of Definition \ref{D:reg}. 
\end{lemma}

In particular, we see that $\Gamma_{\kappa_0}(U)$ is fully contained in $\Gamma(U)$, rather than in the extended free boundary $\Gamma_*(U)$. We further note that  by arguing as in the proof of Lemma~10.5 in \cite{DGPT} for the case $a=0$, we can show that $(x,t) \mapsto \kappa_U(x, t)$
is upper semicontinuous on $\Gamma(U)$. Consequently,  in view of Lemma~\ref{gap1}, we can assert that the following holds.

\begin{lemma}\label{open}
Let $U$ solve \eqref{e0}. Then, the regular set $\Gamma_{\kappa_0}(U)$ is a relatively open subset of the free boundary $\Gamma(U)$. 
\end{lemma}

With the aid of Lemma~\ref{equiv}, we next show that, near a regular point, the free boundary is $\mathbb{H}^{1+\alpha, \frac{1+\alpha}{2}}$ regular for some $\alpha>0$. This will be achieved via   reduction to the elliptic thin obstacle problem satisfied  by $U(\cdot, t)$. Besides Lemma~\ref{equiv}, the other main tool in such reduction  is the epiperimetric inequality established in \cite{GPPS}  which we now describe.

\subsection{An epiperimetric inequality}\label{S:epi}

In order to state the epiperimetric inequality,   we introduce the relevant Weiss type functional. 

\begin{dfn}\label{D:weiss}
Let $v\in W^{1,2}_{\rm loc}(\B_1,|y|^a dX)\cap C(\B_1)$. For a given $\kappa \ge 0$ we  define the $\kappa$-th \emph{Weiss-type functional} $r\to \mathscr W_\kappa(v,r)$ as
\begin{equation}\label{weiss}
\mathscr W_\kappa(v,r) \overset{\rm def}{=}\frac{1}{r^{n+a-1+2\kappa}} I(v,r) -\frac{\kappa}{r^{n+a+2\kappa}} H(v,r).
\end{equation}
\end{dfn}

It is important to observe right away that if $v$ is a $\kappa$-homogeneous solution to \eqref{epb222} we have from \eqref{I} 
\[
I(v,r) =\int_{\partial \BB}v \langle\nabla v,\nu\rangle |y|^a = \frac{\kappa}{r} \int_{\partial \BB} v^2 |y|^a = \frac{\kappa}{r} H(v,r),
\]
where in the last equality we have used Euler's formula $\langle\nabla v,\nu\rangle = r^{-1}\langle\nabla v,X\rangle = \kappa r^{-1} v$. This identity implies that $\mathscr W_{\kappa}(v,r) \equiv 0$ for $0<r \le 1$. 
Also, to provide the reader with an understanding of the powers of $r$ in the definition \eqref{weiss}, we note that the dimension of the measure $|y|^a dX$ is $Q = n+a+1$, and thus $n+a-1+2\kappa = Q-2 + 2\kappa$, whereas $n+a+2\kappa = Q-1+ 2\kappa$. In terms of the dimension $Q$ the powers in \eqref{weiss} are thus in tune with the one-parameter family of Weiss-type functionals introduced in \cite{GP} for the classical Signorini problem (which we recall corresponds to the case  $a = 0$). The reader should also see Theorem~3.8 in \cite{GRO} and Lemma~7.3 in the same paper, where the case of higher homogeneities $\kappa \ge 2$ was treated in the analysis of singular free boundary points.

In the sequel we will be particularly interested in the minimal  homogeneity \eqref{kappazero}.
As a consequence, we have $n+a-1+2\kappa_0 = n+2$, $n+a+2\kappa_0 = n+3$, and the corresponding Weiss-type functional in \eqref{weiss} becomes
\begin{equation}\label{weissk0}
\mathscr W_{\kappa_0}(v,r) = \frac{1}{r^{n+2}} I(v,r) -\frac{\kappa_0}{r^{n+3}} H(v,r).
\end{equation}
It is worth noting that we can also write \eqref{weissk0} in the suggestive form
\begin{equation}\label{weissk0-0}
\mathscr W_{\kappa_0}(v,r) = \frac{H(v,r)}{r^{n+3}}\left(N(v,r) - \kappa_0\right).
\end{equation}
For brevity, we will hereafter drop the subscript $\kappa_0$, and simply write $\mathscr W(v,r)$. 

\begin{thrm}\label{T:weiss}
Let $v$ be a solution to \eqref{epb222} and suppose that $0\in \G(v)$ and that 
\[
\Phi(v,0^+) \ge n + a + 2 \kappa_0 = n+3.
\]
Then, there exist constants $C, r_0>0$,
depending on $n, a$ and $\|f\|_{L^\infty(\B_1)}$, such that for every
$0<r<r_0$ one has
\begin{equation}\label{W'}
\frac{d}{dr}\left(\mathscr W(v,r) + C r^{\frac{1+a}2}\right) \ge \frac{2}{r^{n+2}}
\int_{\p \B_r}\left(\langle\nabla v,\nu\rangle -
\frac{\kappa_0}{r} v\right)^2 |y|^a.
\end{equation}
In particular, the function $r\mapsto
\mathscr W(v,r)+Cr^{\frac{1+a}2}$ is monotone non\-de\-creasing and therefore it has a limit as $r\to 0^+$. Since $a\in (-1,1)$, we conclude that also the following 
limit exists
\[
\mathscr W(v,0^+) \overset{\rm def}{=}  \lim\limits_{r\rightarrow 0^+} \mathscr W(v,r).
\]
\end{thrm}

\begin{proof}
In what follows we will write for brevity $\mathscr W(r)$ instead of $\mathscr W(v,r)$. Using \eqref{H'} in Lemma~\ref{L:H'} and \eqref{fve1}--\eqref{fve2} in Lemma~\ref{L:fve}, and recalling that $2\kappa_0 = 3-a$, we obtain
\begin{align*}
\Wa'(r) & = \frac{1}{r^{n+2}} I'(r) - \frac{n+2}{r^{n+3}} I(r) - \frac{\kappa_0}{r^{n+3}} H'(r) + \frac{\kappa_0(n+3)}{r^{n+4}} H(r)
\\
& = \frac{1}{r^{n+2}} \bigg\{\frac{n+a-1}{r} I(r) + 2 \int_{\p \BB} \langle\nabla v,\nu\rangle^2  |y|^a - \frac 2r \int_{\BB} \langle\nabla v,X\rangle f |y|^a
\\
&\qquad - \int_{\p \BB} vf |y|^a - \frac{n+a-1}{r} \int_{\BB} vf |y|^a \bigg\} - \frac{n+2}{r^{n+3}} I(r)
\\
&\qquad  - \frac{\kappa_0}{r^{n+3}} \bigg\{ \frac{n+a}{r} H(r) + 2 I(r)\bigg\} + \frac{\kappa_0(n+3)}{r^{n+4}} H(r)
\\
& = \left(\frac{n+a-1}{r} - \frac{n+2}{r^{n+3}} - \frac{2\kappa_0}{r^{n+3}}\right) I(r) + \frac{\kappa_0(n+3) - \kappa_0(n+a)}{r^{n+4}} H(r)
\\
&\qquad + \frac{2}{r^{n+2}} \int_{\p \BB} \langle\nabla v,\nu\rangle^2  |y|^a - \frac 2{r^{n+3}} \int_{\BB} \langle\nabla v,X\rangle f |y|^a
\\
&\qquad - \frac{1}{r^{n+2}} \int_{\p \BB} vf |y|^a - \frac{n+a-1}{r^{n+3}} \int_{\BB} vf |y|^a
\\
& = - \frac{2\kappa_0}{r^{n+3}} I(r) + \frac{2 \kappa_0^2}{r^{n+4}} H(r) + \frac{2}{r^{n+2}} \int_{\p \BB} \langle\nabla v,\nu\rangle^2  |y|^a
\\
&\qquad - \frac 2{r^{n+3}} \int_{\BB} \langle\nabla v,X\rangle f |y|^a - \frac{1}{r^{n+2}} \int_{\p \BB} vf |y|^a - \frac{n+a-1}{r^{n+3}} \int_{\BB} vf |y|^a
\\
& = \frac{2}{r^{n+2}} \int_{\p \BB}\left(\langle\nabla v,\nu\rangle - \frac{\kappa_0}{r} v\right)^2 |y|^a - \frac 2{r^{n+3}} \int_{\BB} \langle\nabla v,X\rangle f |y|^a
\\
&\qquad - \frac{1}{r^{n+2}} \int_{\p \BB} vf |y|^a - \frac{n+a-1}{r^{n+3}} \int_{\BB} vf |y|^a.
\end{align*}
At this point we observe that the hypothesis $\Phi(v,0^+) \ge n + a + 2 \kappa_0$ implies  the existence of $r_0>0$ and $\bar C>0$ such that for $r\in (0,r_0)$ we have
\[
H(r) \le \bar C r^{n + a + 2 \kappa_0},\quad G(r) \le \bar C r^{n+a+2\kappa_0 +1}.
\]
This gives
\begin{equation}\label{qui}
\left|\frac{1}{r^{n+2}} \int_{\p \BB} vf |y|^a\right| \le C \|f\|_{L^\infty(\B_1)} r^{\frac{n+a}{2} - n -2} H(r)^{1/2} \le  C \|f\|_{L^\infty(\B_1)} r^{\frac{a-1}{2}}.  
\end{equation}
Similarly, we have
\begin{equation}\label{quo}
 \left|\frac{n+a-1}{r^{n+3}} \int_{\BB} vf |y|^a\right| \le C \|f\|_{L^\infty(\B_1)} r^{\frac{n+a+1}{2} - n -3} G(r)^{1/2} \le C \|f\|_{L^\infty(\B_1)} r^{\frac{a-1}{2}}.
\end{equation}
Finally, we find
\[
\left|\frac 2{r^{n+3}} \int_{\BB} \langle\nabla v,X\rangle f |y|^a\right| \le C \|f\|_{L^\infty(\B_1)} r^{\frac{n+a+1}{2} - n -2} D(r)^{1/2}.
\]
To estimate $D(r)$ we now use \eqref{caccio} in Lemma~\ref{L:caccio}, obtaining
\begin{align*}
D(r) & \le \left(\frac{C}{r^2} + \|f\|_{L^\infty(\B_1)}\right) G(2r) \le C' (1+ \|f\|_{L^\infty(\B_1)}) r^{n+a+2\kappa_0 -1}.
\end{align*}
Substituting this estimate in the above one we find
\begin{equation}\label{qua}
\left|\frac 2{r^{n+3}} \int_{\BB} \langle\nabla v,X\rangle f |y|^a\right| \le C''  r^{\frac{n+a+1}{2} - n -2 + \frac{n+a+2\kappa_0 -1}2} = C'' r^{\frac{a-1}2},
\end{equation}
where $C''>0$ depends on $n, a, \|f\|_{L^\infty(\B_1)}$. Combining \eqref{qui}, \eqref{quo} and \eqref{qua}, we obtain for $r\in (0,r_0)$
\[
\frac{d}{dr} \left(\mathscr W(r) + C r^{\frac{1+a}2}\right) \ge \frac{2}{r^{n+2}}
\int_{\p \B_r}\left(\langle\nabla v,\nu\rangle -
\frac{\kappa_0}{r} v\right)^2 |y|^a.
\]
This is the desired conclusion \eqref{W'}. 
\end{proof}

In the statement of Theorem~\ref{T:epi} below instead of $\mathscr W(v,r)$ we will use the following modified functional
\begin{equation}\label{weissk00}
\mathscr W_{0}(v,r) = \frac{1}{r^{n+2}}\int_{\BB} |\nabla v|^2 |y|^{a} -\frac{\kappa_0}{r^{n+3}}\int_{\partial \BB} v^2 |y|^{a} = \frac{1}{r^{n+2}} D(v,r) - \frac{\kappa_0}{r^{n+3}} H(v,r).
\end{equation}
When $r=1$ we will write $\mathscr W_0(v)$, instead of $\mathscr W_0(v,1)$
We will also need the prototypical function $\hat v_0$ in \eqref{hatv0} with $c>0$.
Since, as we have noted, such $\hat{v}_0$ is a global solution of the problem \eqref{epb222} with  $f\equiv 0$, from \eqref{ID} in Lemma~\ref{L:ID} we have $I(\hat v_0,r) = D(\hat v_0,r)$. Therefore, 
$\mathscr W(\hat v_0,r) = \mathscr W_0(\hat v_0,r)$.  Furthermore, $\hat v_0$ is homogeneous of degree $\kappa_0$, i.e., $\hat v_0(\la X) = \la^{\kappa_0} \hat v_0(X)$. Therefore, from what we have noted above we have in particular 
\[
\mathscr W_0(\hat v_0) = \mathscr W(\hat v_0) = 0.
\]
We mention that the geometric meaning of the functional $\mathscr W_0$ in \eqref{weissk00} is
that it measures the closeness of the 
solution $v$ to the prototypical homogeneous solutions of degree
$\kappa_0$, i.e., the function $\hat v_0$ in \eqref{hatv0}. The following result is  \cite[Theorem~4.2]{GPPS}.

\begin{thrm}[Epiperimetric inequality]\label{T:epi}  There exists
  $\kappa\in (0,1)$ and $\theta \in (0,1)$ such that if $w\in W^{1,2}(\B_1,|y|^a dX)$ is a homogeneous function of degree $\kappa_0 = \frac{3-a}{2}$ such
  that $w\ge 0$ on $B_1$ and $\|w-\hat{v}_0\|_{W^{1,2}(\B_1,|y|^a dX)}\le \theta$, then
  there exists $\tilde{w}\in W^{1,2}(\B_1,|y|^a dX)$ such that $\tilde{w}=w$ on $\partial \B_1$,
  $\tilde{w}$ is nonnegative on $B_1$, and 
\[
\mathscr W_0(\tilde{w})\le (1-\kappa)\mathscr W_0(w).
\]
\end{thrm}

\subsection{Regularity of the   free boundary near regular points}

Let $U$ be the solution of the thin obstacle problem \eqref{e0}. In this subsection we analyze the free boundary of $U$ near a regular point $(x_0, t_0)$. By translation, we may assume  without loss of generality that $(x_0, t_0)=(0,0)$. Since the set of regular points is a relatively open subset of the free boundary, there exists $r_0>0$ such that $\Gamma(U) \cap Q_{r_0}$ consists only of regular points.  We denote this set by $\Gamma_{\kappa_0}$, where $\kappa_0$ is as in \eqref{kappazero}. Now for every $(x,t) \in \Gamma_{\kappa_0}$, we note that $U(\cdot, t)$ solves the elliptic thin obstacle problem \eqref{epb222} with right hand side $f= F+ U_t$ which is uniformly bounded independent of $t$. For a fixed time level $t$, we denote by $\tilde U_{r,x} (\cdot,t)$ the space-like Almgren rescaling of $U(\cdot,t)$ centered at $x$, see \eqref{ar}. We also consider the space-like \emph{homogeneous rescalings}  centered at a point $(x_0, t_0) \in \Gamma_{\kappa_0}$.
\[
U^\star_{r, x_0}(X,t_0)= \frac{U(x_0 + rX, t_0)}{r^{\frac{3-a}{2}}}.
\]
From  the growth estimate \eqref{g1} and the uniform Schauder estimates in Theorem~\ref{schauder}, it follows that the functions $U_{r, x_0}^\star(\cdot,t_0)$ are uniformly bounded in $C_{a,{\rm loc}}^{1+\gamma}$, for some $\gamma>0$ independent of $(x_0,t_0)$.  Similarly, Theorem~\ref{conv} implies that $\tilde U_{r, x}$ are uniformly  bounded in $C_{a, loc}^{1+\gamma}$ and converge to  some $v_0$ on a subsequence $r_j \to 0$. Moreover, after a rotation of coordinates, we have that such a $v_0$  has the form \eqref{hatv0}.   

Recalling that $0<\frac{1-a}{2}<1$, let $1>\delta>\frac{1-a}{2}$. As in \eqref{generalized-frequency}, consider the functional $\Phi_{\delta,x} (U(\cdot,t),r)$ corresponding to the free boundary point $x \in \Gamma(U(\cdot,t))$. Using the regularity estimates in Theorem~\ref{localized} and Theorem~\ref{T:ellalmgren}, we can  argue as in the proof of \cite[Lemma~2.5]{GPPS} to conclude that as $r \to 0^+$,
\begin{equation}\label{uni1}
\Phi_{\delta,x} (U(\cdot,t),r) \to n+ a+ 2\kappa_0,
\end{equation}
uniformly in $(x, t) \in \Gamma_{\frac{3-a}2} \cap Q_{r_1}$, 
provided $r_1$ is small enough. As in \cite{GPPS}, we denote by $\mathcal{H}_{\frac{3-a}{2}}$, the space of $\frac{3-a}{2}$ homogeneous functions  of the form
\[
  c\left(\langle x, e\rangle+\sqrt{ \langle x, e\rangle ^2+y^2}\right)^{\frac{1-a}{2}}\left( \langle x, e\rangle -\frac{1-a}{2}\sqrt{\langle x, e\rangle^2+y^2}\right),
 \]
 for some $c>0$ and $|e|=1$. From a compactness argument as in the proof of \cite[Lemma~3.4]{GPG}, the uniform Schauder estimates in Theorem~\ref{schauder}, and from \eqref{uni1}, it follows that given
$\theta>0$, there exists $r_1>0$ such that
\begin{equation}\label{uni2}
\inf_{h \in \mathcal{H}_{\frac{3-a}{2}}} \| \tilde U_{r,x} (\cdot, t) - h\|_{C_a^{1+\gamma}(\overline{\B_{1}^+})} < \theta
\end{equation}
for all $r< r_1 $ and $(x,t) \in Q_{r_1} \cap \Gamma_{\kappa_0}$.  This shows that $\tilde U_{r,x}(\cdot,t)$ satisfies the hypothesis of  Theorem~\ref{T:epi}. Since the homogeneous rescalings $U_{r, x}^\star(\cdot,t)$ are constant multiples of the Almgren rescalings, the same also holds for $U_{r,x}^\star(\cdot, t)$. 

Now using  the Weiss type monotonicity result  in Theorem~\ref{T:weiss}, the scaling properties of the Weiss functional,  the boundedness of $U_t$,  Theorem~\ref{T:epi},  the growth estimate in \eqref{g1} and the uniform Schauder estimate in Theorem~\ref{schauder}, we can argue as in the proofs of  Lemma~3.7, Lemma~3.8,  Lemma~5.1, Proposition~5.2 and 5.3 in \cite{GPPS}  and consequently assert that  
\begin{equation}\label{con4}
\begin{cases}
U_{r,x}^\star(\cdot, t) \to U_{(x, t), 0}\text{ in $C_a^{1+\gamma}(\overline{\B_{1}^+})$, where $U_{(x,t),0} \in \mathcal{H}_{\frac{3-a}{2}}$,}
\\
\text{$U_{(x,t),0}$ is nonzero,}
\\
\displaystyle\int_{\partial \B_1} |U_{r,x}^\star(\cdot, t) - U_{(x,t), 0} | |y|^a \leq Cr^{\beta}\ \text{for some $\beta>0$},
\end{cases}
\end{equation}
for  all $r < r_1$ and $ (x,t) \in Q_{r_2} \cap \Gamma_{\kappa_0}$, where $r_1$ and $r_2$ are sufficiently small and where $C$ is some universal constant. 
We now show that for some $\gamma>0$, we have that
\begin{equation}\label{close}
\int_{\partial B_1} |U_{(x_1, t_1), 0} - U_{(x_2,t_2), 0} |  \leq C(|x_1 - x_2| + |t_1 - t_2|^{1/2})^{\gamma}
\end{equation}
Since $U_{(x_i, t_i), 0}$ (for $i=1,2$) are $\frac{3-a}{2}$ homogeneous functions,  it suffices to show that
\begin{equation}\label{close1}
\int_{ B_1} |U_{(x_1, t_1), 0} - U_{(x_2,t_2), 0} |  \leq C(|x_1 - x_2| + |t_1 - t_2|^{1/2})^{\gamma}.
\end{equation}
From the last estimate  in \eqref{con4}  above we obtain 
\begin{equation}\label{close2}
\int_{\partial \B_1} |U_{(x_1, t_1), 0} - U_{(x_2,t_2), 0} | |y|^a \leq Cr^{\beta} + \int_{\partial \B_1} |U_{r,x_1}^{\star}(\cdot, t_1) - U_{r,x_2}^{\star}(\cdot, t_2) | |y|^a.
\end{equation}
Applying the mean value theorem we infer
\begin{equation}\label{mv}
\begin{aligned}
&|U_{r,x_1}^{\star}(\cdot, t_1) - U_{r,x_2}^{\star}(\cdot, t_2) | \\
&\qquad= \frac{1}{r^{\frac{3-a}{2}}} \int_{0}^1  \nabla_{x} U( bx_1 + (1-b)x_2 + rx, ry, bt_1 + (1-b) t_2). (x_1 - x_2) db 
\\
&\qquad\qquad+ \frac{1}{r^{\frac{3-a}{2}}} \int_{0}^1  \partial_t U( bx_1 + (1-b)x_2 + rx, ry, bt_1 + (1-b) t_2). (t_1 - t_2) db.
\end{aligned}
\end{equation}
The first estimate in \eqref{k2}  centered at the free boundary point $(x_1, t_1)$ gives, for $0<b<1$,
\begin{equation}\label{grad}
\begin{aligned}
&|\nabla_{x} U( bx_1 + (1-b)x_2 + rx, ry, bt_1 + (1-b) t_2)|
\\
&\qquad=|\nabla_{x} U( x_1 + (1-b) (x_2-x_1) + rx, ry, t_1 + (1-b)( t_2- t_1))|
\\
&\qquad\leq C \left( |x_1 - x_2| + |t_1-t_2|^{1/2} +r ( |x| + |y| ) \right) ^{\frac{1-a}{2}}.
\end{aligned}
\end{equation}

 Using the boundedness of  $U_t$, we obtain from \eqref{mv} and \eqref{grad}  
\begin{equation}\label{interm}
\begin{multlined}
\int_{\partial \B_1} |U_{r,x_1}^{\star}(\cdot, t_1) - U_{r,x_2}^{\star}(\cdot, t_2) | |y|^a \\\leq C\biggl(  \biggl(\frac{(|x_1 - x_2| + |t_1- t_2|^{1/2})}{r}\biggr)^{\frac{3-a}{2}} + \frac{|x_1 - x_2|}{r}  + \frac{|t_1- t_2|}{r^{\frac{3-a}{2}}} \biggr).
\end{multlined}
\end{equation}
We now let $r= ( |x_1 - x_2| + |t_1 - t_2|^{1/2})^{\sigma}$ for some $\sigma \in (0,1)$ arbitrarily fixed. With this choice of $r$, combining \eqref{close2} and \eqref{interm}, we obtain  
\begin{equation}\label{close4}
\int_{\partial \B_1} |U_{(x_1, t_1), 0} - U_{(x_2, t_2), 0} | |y|^a \leq  C( |x_1 - x_2| + |t_1 - t_2|^{1/2})^{2\gamma},
\end{equation}
for some $\gamma$ depending also on $\sigma$,
Given \eqref{close4}, we can proceed as in the proof of Proposition~5.4 in \cite{GPPS} and thus we can  finally assert that \eqref{close} holds.  From now on, we will denote the blow up limit 
\begin{equation}\label{blowup1}
U_{(x,t), 0}= c_{(x,t)}\left(\langle x, e_{(x,t)}\rangle+\sqrt{ \langle x, e_{(x,t)}\rangle ^2+y^2}\right)^{\frac{1-a}{2}}\left(\langle x, e_{(x,t)}\rangle -\frac{1-a}{2}\sqrt{\langle x, e_{(x,t)}\rangle ^2+y^2}\right). 
 \end{equation}
Arguing as in  the proof of Lemma~5.6 in \cite{GPPS}, one can deduce from \eqref{close}  that  the following inequalities hold:
 \begin{equation}\label{close5}
 \begin{aligned}
 |c_{(x_1, t_1)} - c_{(x_2, t_2)}| &\leq C( |x_1 - x_2| + |t_1 - t_2|^{1/2})^{\gamma},
 \\
 |e_{(x_1, t_1)} - e_{(x_2, t_2)}| &\leq C( |x_1 - x_2| + |t_1 - t_2|^{1/2})^{\gamma}. 
 \end{aligned}
 \end{equation}
 With these estimates in hand, we now proceed with the proof of regularity of the free boundary $\Gamma(U)$ near $(0,0)$.
 
 \medskip
 
 \emph{Step 1:}  We first note that the boundedness of $U_t$ and the  uniform Schauder estimates in Theorem~\ref{schauder} imply that  $\|U_{r,x}^{\star}(\cdot, t)\|_{C_a^{1+\alpha}(\overline{\B_1^+})}$ are uniformly bounded independent of $(x,t) \in \Gamma_{\kappa_0} \cap Q_{r_2}$ for $r < r_1$ and some $\alpha>0$. Then, using the third estimate in \eqref{con4}, estimates in \eqref{close5} and a compactness argument as in Step 1 in the proof of Theorem~1.2 in \cite{GPPS}, it is possible to show that given $\ve>0$, there exists $r_\ve >0$ such that for $r< r_\ve$, we have
 \begin{equation}\label{close7}
 \|U_{r,x}^{\star} (\cdot, t) - U_{(x,t), 0} \|_{C_a^{1+\alpha}(\overline{\B_1^+})} \leq  \ve
 \end{equation}
 for $(x,t) \in \Gamma_{\kappa_0} \cap Q_{r_2}$. 
 
 \medskip
 
 \emph{Step 2 (Conclusion):} Without loss of generality, we may assume that $e_{(0,0)}= e_n=(0,...0, 1)$. Given the weighted $C^{1}$ estimate  \eqref{close7},  we can now  repeat the arguments  as in Step 2-Step 4 in the proof of Theorem~1.2 in \cite{GPPS} to conclude that  for  a given $\ve$ small enough, there exists $ r_\ve, r_2$ small enough such that for $(x,t) \in \Gamma_{\kappa_0} \cap Q_{r_2}$, 
\begin{equation*}
\begin{aligned}
x+ ( \mathcal{C}_{\ve}( e_n)  \cap B_{r_\ve})&\subset  \{U(\cdot, t) >0\},
\\
x- ( \mathcal{C}_{\ve}( e_n)  \cap B_{r_\ve})&\subset  \{U(\cdot, t) =0\}.
\end{aligned}
\end{equation*}
Here,  for a unit vector $e$, $\mathcal{C}_{\ve}(e)$ is the cone in $\Rn$ defined by 
\begin{equation}\label{cone}
\mathcal{C}_{\ve}(e)= \{(x', x_n)\mid \langle x, e\rangle \geq \ve |x|\}
\end{equation}
Fixing $\ve=\ve_0$,  this then implies in a standard way  that  for every $t$ with $|t| \leq r_2^2$, there exists a Lipschitz function $g(\cdot, t): \R^{n-1} \to \R$, with Lipschitz norm bounded by $\frac{C}{\ve_0}$,   such that the free boundary $\Gamma_{\kappa_0} \cap Q_{r_2}$ can be represented as $\{x_n=g(x', t)\}$.  Moreover, we also have that  for all $(x_0,t_0) \in \Gamma_{\kappa_0} \cap Q_{r_2}$ and $r_2$ small,  
\begin{equation}\label{St}
\begin{aligned}
\{x_n \leq g(x',t)\}  \cap B_{r_2}(x_0) &= \{U(\cdot, t_0) =0\} \cap B_{r_2}(x_0),
\\
\{x_n > g(x', t)\}  \cap B_{r_2}(x_0) &= \{U(\cdot, t) >0\} \cap B_{r_2}(x_0).
\end{aligned}
\end{equation}
Moreover by letting $\ve \to 0$, we see that $\Gamma(U(\cdot, 0))$ is differentiable at $x=0$ with normal $e_n$. This in turn implies the space like differentiability of $g$ at $0$. Likewise $g(\cdot, t)$ is differentiable with respect to $x' \in \R^{n-1}$ at every $x'$ such that $(x', g(x', t), t) \in \Gamma_{\kappa_0} \cap Q_{r_2}$. It also follows that $\Gamma(U(\cdot, t))$ has a normal $e_{(x,t)}$ at $(x,t) \in \Gamma_{\kappa_0} \cap Q_{r_2}$.  Using the fact that $(x,t) \to e_{(x,t)}$ is in $H^{\gamma, \frac{\gamma}{2}}$ which follows from \eqref{close5}, we obtain that $\nabla_{x'}g$ is in $H^{\gamma, \frac{\gamma}{2}}$.  

\medskip

We now make the following claim. 

\emph{Claim:} For a possibly smaller $r_2$,  the following nondegeneracy estimate holds: 
\begin{equation}\label{nondeg}
|U(x,0,  t_0)| \geq c |x_n - g(x', t_0)|^{\frac{3-a}{2}}\ \text{whenever   $t_0 \leq r_2^2$ and $x_ n> g(x', t_0)$},
\end{equation}
for some $c$ universal independent of $t_0$. 

Before proving the claim, we show that such a nondegeneracy estimate  implies that $g$ is in fact H\"older continuous in $t$ with exponent $\frac{1+\alpha}{2}$ for some $\alpha>0$. This would then imply that $\Gamma_{\kappa_0} \cap Q_{r_2}$ is $H^{1+\alpha,\frac{1+\alpha}{2}}$-regular for a possibly different $\alpha>0$, depending also on $\gamma$ above. 

\medskip

Indeed from \eqref{nondeg} and the boundedness of $U_t$ ( say $|U_t| \leq M$) it follows 
\[
U(x,0,t) \geq c |(x_n- g(x',t_0)|^{\frac{3-a}{2}} - M |t-t_0|.
\]
Taking $x_n- g(x',t_0) = r$,  we have
\begin{equation}\label{pos}
U(x', g(x',t_0)+ r,0,  t) \geq c r^{\frac{3-a}{2}}- M |t-t_0| >0,
\end{equation}
provided 
\[
M |t -t_0| \leq c r^{\frac{3-a}{2}}.
\]
Now  \eqref{pos}  and   \eqref{St}  imply that 
\[
g(x',t)< x_n= g(x',t_0) + r
\]
Letting 
\[
c r^{\frac{3-a}{2}} = 2 M|t- t_0|,
\]
we obtain 
\[
g(x',t)< g(x',t_0)+ C |t -t_0|^{\frac{2}{3-a}}
\]
for a different $C$. Interchanging  $t$ and $t_0$, we thus conclude
\[
|g(x',t) - g(x',t_0)| < C |t -t_0|^{\frac{2}{3-a}}.
\]

Since $a \in (-1, 1)$, we have $\frac{2}{3-a} > 1/2$, and thus the $\frac{1+\alpha}{2}$-H\"older continuity of $g$ in $t$ follows.  This would then imply the $H^{1+\alpha, \frac{1+\alpha}{2}}$-regularity of the free boundary near the regular point $(0, 0)$.  In order to conclude,  we are  now only left  with proving  \eqref{nondeg}. 

\medskip

\emph{Proof of Claim:} For $(x^1, t^1) \in \Gamma_{\kappa_0} \cap Q_{r_2}$,  we can start from \eqref{blowup1}  and \eqref{close5} and proceed as in Step 2- Step 4 in the proof of Theorem~1.2 in \cite{GPPS}     to obtain that for a given $\ve_0>0$ fixed ,  and $r_2$ small enough depending also on $\ve_0$, 
\[
(\mathcal{C}_{\ve_0}(e_{n})  \cap B_{r_1} ) + x^1 \subset \{U(\cdot, t^1) >0 \}
\]
for $r_1$ small enough, whenever $(x^1,t^1) \in \Gamma_{\kappa_0} \cap Q_{r_2}$. Moreover on  $K_{\ve_0}(e_n)= \mathcal{C}_{\ve_0}(e_n) \cap \partial B_1$, we can also ensure 
\[
U_{(x^1,t^1), 0}  > a_0>0
\]
for some $a_0>0$ universal depending on $\ve_0$.  From the estimate \eqref{close7} (with $\ve= \frac{a_0}{2}$), we obtain that for $r \leq r_1$ for a possibly smaller $r_1$ depending also on $a_0$  that  the following holds, 
\begin{equation}\label{bl}
U_{r, x^1}^{\star}(\cdot, t^1) > a_0/2\ \text{on $K_{\ve_0}(e_n)$}
\end{equation}
for all $(x^1,t^1) \in \Gamma_{\kappa_0} \cap Q_{r_2}$.  Now given $(x^1, t^1)= (x', g(x', t^1), t^1) \in \Gamma_{\kappa_0}$,  let $(x,t^1)= (x', x_n, t^1)$ be such that $x_n > g(x', t^1)$. Then by letting $r= x_n - g(x', t^1)$, we obtain from \eqref{bl} that 
\begin{equation}\label{bl1}
U_{r,x^1}^{\star} (e_n,0,  t^1) = \frac{U(x', g(x', t^1) + r, 0,  t^1)}{r^{\frac{3-a}{2}}} =\frac{U(x', x_n, t^1)}{ |x_n - g(x', t^1)|^{\frac{3-a}{2}}} \geq  a_0/2
\end{equation}
from which \eqref{nondeg} follows. 
 We have thus proved the following.

\begin{thrm}\label{regofreg}
Let $U$ be a solution to \eqref{e0} where $F$ satisfies the bounds as in \eqref{fbound}--\eqref{ptfbound} for some $\ell \geq 4$.  Let $(x_0, t_0)\in\Gamma_{\kappa_0}(U)$. Then there exists a small $r>0$  such that $\Gamma(U) \cap Q_{r}(x_0, t_0)$ is a $H^{1+\alpha, \frac{1+\alpha}{2}}$ graph for some $\alpha>0$. 
\end{thrm}

\appendix

\section{}\label{S:appB}

In this appendix we establish some auxiliary results on the regularity of even and odd solutions to the free equation
\[ 
\La U = |y|^a f
\]
We note that such estimates were  crucially used to establish Theorem~\ref{localized}. 

\subsection{Regularity of even solutions}
We first consider the case of  symmetric (even)  solutions to
\begin{equation}\label{nhom}
\La U= |y|^a f\quad\text{in }\Q_1.
\end{equation}
 When $f \equiv 0$,  we can assert that such solutions   are twice differentiable and hence  are classical solutions. 

\begin{lemma}\label{smooth}
Let $U$ be a solution to 
\begin{equation}\label{zr}
\La U= 0\quad\text{in }\Q_1,
\end{equation}
with $U(x,y,t)= U(x, -y, t)$. Then $U \in H^{2 + \beta, \frac{2+\beta}{2}} (\Q_r)$ for all $r <1 $ and some $\beta>0$. \end{lemma}
\begin{rmrk}
In Lemma~\ref{smooth} above, it seems possible to assert that $U$ is in fact smooth up to $\{y=0\}$ by a bootstrap argument as in the proof of  Lemma~7.6 and 7.7 in \cite{STV}. We however don't address  such a  higher regularity result over here because for our purpose,  such a $H^{2 + \beta, \frac{2+\beta}{2}}$ regularity result for symmetric solutions suffices. 

\end{rmrk}

\begin{proof}
We follow the approach as in \cite{STV} for the elliptic case. We first note that from the De Giorgi-Nash-Moser theory for such degenerate parabolic equations as in \cite{CSe}, we have that $U$ is H\"older continuous. Then by using the translation variance of the equation in $x, t$, we can assert that $U$ is smooth in $x, t$ up to $\{y=0\}$. This later fact can be established by a repeated difference quotient type argument as in Section~5 in \cite{BG}.  Next we also have that $w=y^a U_y$ is a weak solution to the conjugate PDE 
\[
\mathscr{L}_{-a} w=0
\]
and thus $y^a U_y$ is H\"older continuous up to $\{y=0\}$ again by the results in \cite{CSe}. Now from \eqref{zr} it follows that 
\begin{equation}\label{od}
y^{-a} \partial_y ( y^a U_y) = U_{yy} + \frac{a}{y} U_y = - \Delta_x U + U_t = g(X, t).
\end{equation}
Thus, $\mathcal F=U_{yy} + \frac{a}{y} U_y$ is smooth in $x, t$. Moreover, $\mathcal F$ is H\"older in $X,t$ up to $\{y=0\}$.  Now since $U$ is even in $y$, we can restrict it in $\Q_1^+$ and express using \eqref{od} in the following way, 
\begin{equation}
y^a U_y(X,t)= \int_{0}^y z^a g(x,z, t) dz.
\end{equation}
Hence, 
\begin{equation}\label{tg}
\mathcal G=\frac{1}{y} U_y= \frac{1}{y^{1+a}} \int_{0}^y z^a (g(x,z, t)- g(x, 0, t)) dz + \frac{g(x, 0, t)}{1+a}.
\end{equation}
The H\"older continuity of $\mathcal G $ in $x,t$ now follows from the H\"older continuity of $g$. By an exact analogous argument as in the proof of Lemma~7.5 in \cite{STV} using the expression for $\mathcal G$ as in \eqref{tg}, we obtain the H\"older continuity of $\mathcal G$ in $y$. From the H\"older continuity of $\frac{a}{y} U_y, g$ and \eqref{od} if follows that $U_{yy}$ is H\"older continuous up to $\{y=0\}$. This then implies  that $\Delta_X U - U_t$ is H\"older continuous up to $\{y=0\}$. Moreover  since $U$ restricted to $\{y=0\}$ is smooth, therefore  the conclusion of the lemma follows from classical boundary Schauder theory for the heat operator.
\end{proof}

Our next result provides ``almost'' Lipschitz estimate with respect to the parabolic distance  for $\nabla_X V$ when $V$ is a symmetric solution to the nonhomogeneous equation  \eqref{nhom} with bounded $f$. 

\begin{prop}\label{ac1} 
Let $V\in W^{1,2}(\Q_1,|y|^a dXdt)$ be a weak solution of
\begin{equation}\label{eq:Labdd}
\La V=|y|^a f\quad\text{in }\Q_1,
\end{equation}
with $V(x,-y,t) = V(x,y,t)$. Then, for any 
$\alpha\in (0,1)$ and $r< 1$ we have $\nabla_X  V\in H^{\alpha} (\Q_r)$ 
 and moreover the following estimate holds, 
\begin{equation}\label{K}
\|\nabla_X V\|_{H^{\alpha, \alpha/2}(\Q_r)}\leq C(r,a,n, \alpha) (\\|y|^{a/2} V\|_{L^2(\Q_1)}+\|f\|_{L^\infty(\Q_1)}).
\end{equation}
\end{prop}

\begin{rmrk}
We stress that, even with a right-hand side $f\equiv 0$, Proposition~\ref{ac1} fails to be true if we remove the assumption that $V(x,-y,t) = V(x,y,t)$. The function $V(x,y) = |y|^{-a} y$ belongs to $W^{1,2}(\B_1,|y|^a dX)$, and is a weak solution to the stationary equation $L_a V = 0$ in $\B_1$. Its weak derivative $V_y(X) = (1-a) |y|^{-a}$ is not continuous in $\B_1$, unless $-1<a\le 0$.
\end{rmrk}

The proof of Proposition~\ref{ac1} will be based on some preliminary results.
We begin with the following compactness  lemma.

\begin{lemma}\label{L:com}
Let $V$ be a weak solution to \eqref{eq:Labdd} such that
$\|V\|_{L^2(\Q_1,|y|^a dXdt)}\le 1$ and $V(x,-y,t) = V(x,y,t)$. For any
$\ve>0$ there exists $\delta=\delta(\ve,n,a)>0$ such that if
$\|f\|_{L^\infty(\Q_1)} \le \delta$, then there exists a solution
$V_0$ to $\La V_0=0$ in $\Q_{1}$ such that $\|V_0\|_{L^2(\Q_1, |y|^a dX dt)}\leq 1$
and $V_0(x,-y,t) = V_0(x,y,t)$, with
\begin{equation}\label{con}
\|V - V_0\|_{L^\infty(\Q_{1/2})} \le \ve.
\end{equation}
\end{lemma}

\begin{proof}
We argue by contradiction and assume the existence of $\ve_0>0$ such that for every  $k\in\mathbb{N}$ there exist $V_k\in W^{1,2}(\Q_1,|y|^a dX)$ and $f_k\in L^\infty(\Q_1)$ such that
\begin{equation}\label{uk}
\begin{cases}
  L_a V_k = |y|^a f_k\text{ in }\Q_1\\
\|V_k\|_{L^2(\Q_1,|y|^a dXdt)}\leq 1,
\\
\|f_k\|_{L^{\infty}(\Q_1)} \leq \frac 1k,
\end{cases}
\end{equation}
but for which we have for every solution $W$ of $\La W=0$ in $\Q_{1}$
such that $\|W\|_{L^2(\Q_1,|y|^a dXdt)}\leq 1$ and $W(x,-y,t) = W(x,y,t)$,
\begin{equation}\label{contr}
\|V_k-W\|_{L^{\infty}(\Q_{1/2})}\geq \ve_0.
\end{equation}
We will show that \eqref{contr} leads to a contradiction.

Now from the H\"older regularity result in \cite{CSe}, we see that there exists $\beta = \beta(n,a)\in (0,1)$ such that for every $k\in \mathbb N$ one has
\[
[V_k]_{H^{\beta, \beta/2}(\Q_\rho)} \le C(n,a,\rho).
\]
By the theorem of Ascoli-Arzel\`a we can extract a subsequence, which
we keep denoting $\{V_k\}_{k\in \mathbb N}$, and a function $V_0\in
H^{\beta, \beta/2}(\Q_{1})$, such that $V_k \to V_0$ uniformly on compact
subsets of $\Q_{1}$. Suppose we can prove that 
\begin{equation}\label{con2}
V_0\in W^{1,2}_{\rm loc}(\Q_{1},|y|^a dXdt),\quad\text{and}\quad \La V_0 = 0\quad\text{in}\ \Q_{1}.
\end{equation}
Since we clearly have $V_0(x,-y,t) = V_0(x,y,t)$, and also
$\|V_0\|_{L^2(\Q_1,|y|^a dXdt)}\leq 1$ by Fatou's lemma, 
\eqref{con2} leads to a contradiction since by \eqref{contr} and
uniform convergence, we would have 
\[
0 < \ve_0 \le \|V_k-V_0\|_{L^{\infty}(\Q_{1/2})}\to0,\quad\text{as }k\to\infty.
\]
To establish \eqref{con2} we note that by for any
ball $\Q_\rho$, $0<\rho<1$, by the Caccioppoli inequality in \cite{CSe} and by \eqref{uk} again, we infer that 
\[
\|V_k\|_{W^{1,2}(\Q_\rho,|y|^a dXdt)} \le C(n,a,\rho).
\]
Therefore, possibly passing to a subsequence, we have
\[
V_k\to V_0\quad\text{weakly in}\ W^{1,2}(\Q_\rho,|y|^a dXdt)
\]
and consequently $\La V_0 = 0$ in $\Q_\rho$ by passing to the limit in
\eqref{uk}. Thus, \eqref{con2} holds, and the proof of the lemma is complete.
\end{proof}

\begin{cor}\label{C:com}
Let $V$ be a weak solution to \eqref{eq:Labdd} such that
$\|V\|_{L^2(\Q_1,|y|^a dXdt)}\le 1$ and $ V(x,-y,t) = V(x,y,t)$. Given any
$\alpha\in (0,1)$ and $\mu>0$ there are constants $\delta, \la>0$,
depending only on $n, a$ and $\alpha$, and $\mu$ such that if
\begin{equation}\label{cor1}
\|f\|_{L^\infty(\Q_1)} \le \delta,
\end{equation}
then there exists an affine function in the $x$-variable, $\ell(x)$, such that
\begin{equation}\label{cor2}
\|V - \ell\|_{L^\infty(\Q_\la)} \le \mu \la^{1+\alpha}.
\end{equation}
\end{cor}

\begin{proof}
By Lemma~\ref{L:com} for any $\ve>0$ there exists $\delta>0$ such that
for any given $f$ satisfying \eqref{cor1} there exists a weak solution
$V_0$ of $\La v_0=0$ such that $\|V_0\|_{L^2(\Q_1,|y|^a dXdt)}\leq M$, $V_0(x,-y,t) = V_0(x,y,t)$, and 
\begin{equation}\label{cor3}
\|V - V_0\|_{L^\infty(\Q_{1/2})} \le \ve.
\end{equation}
In view of Lemma~\ref{smooth} we know that $V_0\in
H^{2+\beta, \frac{2+\beta}{2}}(\overline{\Q_{1/2}})$, with $\|V_0\|_{H^{2+\beta, \frac{2+\beta}{2}}(\overline{\Q_{1/2}})}
\le C(n,a)$. In particular, by setting $\ell(x)=V_0(0)+\langle
\nabla_x V_0(0),x\rangle$ and using that $\partial_yV_0(0, 0)=0$, we have
that for any $0<\lambda\leq 1/2$
$$
\|V_0(X,t)-\ell(x)\|_{L^\infty(\Q_{\lambda})}\leq C(n,a)\lambda^2.
$$
Consequently,
\[
\|V - \ell\|_{L^\infty(\Q_\la)} \le \|V - V_0\|_{L^\infty(\Q_\la)} + \|V_0 - \ell\|_{L^\infty(\Q_\la)} \le \ve + C(n,a) \la^2.
\] 
If we now choose $\la>0$ such that $C(n,a) \la^2 = \la^{1+\alpha}/2$ and $\ve = \la^{1+\alpha}/2$ we reach the desired conclusion.
\end{proof}

We are now ready to provide the 

\begin{proof}[Proof of Proposition~\ref{ac1}]
It will be sufficient to prove the proposition in the case
$Q_{1/8}$ as the general case will follow by a covering argument.

Let now $V$ be as in the statement of the proposition and let $\alpha\in
(0,1)$ be arbitrary. Let $\mu=\mu(n,a)=1/\|1\|_{L^2(\Q_1,|y|^a dXdt)}$. Denote by $\delta, \la$ the numbers in the statement of
Corollary \ref{C:com}. By dividing $V$ by $\|V\|_{L^2(\Q_1,|y|^a
  dXdt)}+\delta^{-1}\|f\|_{L^\infty(\Q_1)}$, we may assume without loss
of generality that
\begin{equation}\label{cor4}
\|V\|_{L^2(\Q_1,|y|^a
  dXdt)}\leq 1,\quad \|f\|_{L^\infty(\Q_1)}\leq\delta.
\end{equation}
By \eqref{cor2} in Corollary \ref{C:com} we thus find an affine function in $x$, $\ell(x)$, such that
\begin{equation}\label{cor3-1}
\|V - \ell\|_{L^\infty(\Q_\la)} \le \mu \la^{1+\alpha}.
\end{equation}
We now claim that for every $k\in \mathbb N$ there exists $\ell_k$ affine in $x$ such that
\begin{equation}\label{cor5}
\|V - \ell_k\|_{L^\infty(\Q_{\la^k})} \le \mu\la^{k(1+\alpha)}.
\end{equation}
We prove the claim \eqref{cor5} by induction. By taking $\ell_1 = \ell$, it is clear that \eqref{cor5} is true when $k=1$. Assume it is true for a certain $k\ge 1$, and thus exists $\ell_k$ as in \eqref{cor5}. We want to show that it is true also for $k+1$. Let
\[
V_k(X,t) = \frac{V(\la^k X, \la^{2k} t) - \ell_k(\la^k x)}{\la^{k(1+\alpha)}},\quad (X,t)\in \Q_1.
\]
By the first estimate in \eqref{cor5} we know that
$\|V_k\|_{L^\infty(\Q_1)} \le \mu$ and consequently
$\|V_k\|_{L^2(\Q_1,|y|^a dXdt)}\leq 1$, by the choice of $\mu$.
Furthermore, since $L_a(\ell_k) = 0$ we have
\[
\La V_k = \la^{2k - k(1+\alpha)} |y|^a f(\la^k X, \la^{2k} t) \overset{\rm def}{=} |y|^a f_k(X,t). 
\]
Since $\alpha<1$ and $\la \le 1$, we have $\la^{2k - k(1+\alpha)}\le
1$ and therefore by \eqref{cor4} 
\[
\|f_k\|_{L^\infty(\Q_1)} \le \delta,
\]
and thus $V_k$ and $f_k$ satisfy the assumptions in Corollary~\ref{C:com}. As a consequence, there exists an affine function in $x$, $\tilde \ell_k(x)$, such that
\begin{equation}\label{cor6}
\|V_k - \tilde\ell_k\|_{L^\infty(\Q_\la)} \le \mu \la^{1+\alpha}.
\end{equation}
If we let
\[
\ell_{k+1}(x) = \ell_k(x) + \la^{k(1+\alpha)} \tilde \ell_k(\la^{-k} x),
\]
then for $(X,t)\in \Q_{\la^{k+1}}$ the point $(X',t' )= (\la^{-k} X, \la^{-2k} t)\in \Q_\la$ and we obtain from \eqref{cor6}
\begin{align*}
V(X,t) - \ell_{k+1}(x) & = V(\la^k X', \la^{2k} t') - \ell_{k+1}(\la^k x') \\
&= V(\la^k X', \la^{2k} t') - \ell_k(\la^k x') - \la^{k(1+\alpha)} \tilde \ell_k(x')
\\
& = \la^{k(1+\alpha)} \left(V_k(X', t') - \tilde\ell_k(x')\right).
\end{align*}
Therefore, by \eqref{cor6}
\[
\|V - \ell_{k+1}\|_{L^\infty(\Q_{\la^{k+1}})} = \la^{k(1+\alpha)} \|V_k - \tilde \ell_k\|_{L^\infty(\Q_\la)} \le  \mu\la^{(k+1)(1+\alpha)}.
\]
We have thus verified that  \eqref{cor5} holds for $k+1$. We further
note that \eqref{cor5} implies that
$$
\|\ell_k-\ell_{k+1}\|_{L^\infty(\Q_{\lambda^k})}\leq 2\lambda^{k(1+\alpha)}
$$
and therefore
$$
|\ell_k(0)-\ell_{k+1}(0)|\leq 2\mu\lambda^{k(1+\alpha)},\quad |\nabla
\ell_k(0)-\nabla\ell_{k+1}(0)|\leq 4\mu\lambda^{k\alpha}.
$$
In particular, there exists a limit $\ell_0(x)$ of the affine functions
$\ell_k$, as $k\to \infty$, and
$$
|\ell_k(0)-\ell_0(0)|\leq C\lambda^{k(1+\alpha)},\quad |\nabla
\ell_k(0)-\nabla\ell_0(0)|\leq C\lambda^{k\alpha}
$$
for $C=C(n,a,\alpha)$. Hence, we can conclude that
 by a standard argument
\begin{equation}\label{eq:v-affine-approx}
|V(X,t)-\ell_0(x)|\leq C|(X,t)|^{1+\alpha},\quad (X,t)\in\Q_{1/2}.
\end{equation}
where $|(X_1, t_1)- (X_2, t_2)|= |X_1-X_2| + |t_1-t_2|^{1/2}$.
In particular, $V$ is differentiable at the origin and we can write
that
$$
\ell_0(x)=V(0,0)+\langle\nabla_xV(0,0),x\rangle.
$$
Now, if we denote
$$
W(X,t)=V(X,t)-\ell_0(x),\quad (X,t)\in\Q_1,
$$
then by \eqref{eq:v-affine-approx} we will have
\begin{equation}\label{K12}
\sup_{\Q_{r}} |W|\leq C r^{1+\alpha},\quad 0<r\leq 1/2,
\end{equation}
with $C=C(n,a,\alpha)$.
Consider now the homogeneous rescalings of order $2$
$$
\tilde W_r(X,t)=\frac{W(rX, r^2 t)}{r^2},
$$
which satisfy
$$
\La (\tilde W_r)(X,t) = |y|^a f(rX,r^2t)\quad\text{in }\Q_{1/r}.
$$
We also have from \eqref{K12}
$$
\sup_{\Q_1} |\tilde W_r|\leq C r^{\alpha-1}.
$$
Besides, we can apply interior estimates to obtain that for $t_0 \in (-1, 0)$, 
\newcommand{\osc}{\operatornamewithlimits{osc}}
$$
\sup_{\Q_{1/4}((1/2)e_{n+1}, t_0)}|\nabla \tilde W_r|+ [\nabla \tilde
W_r]_{H^{\alpha, \alpha/2}(\Q_{1/4}((1/2)e_{n+1}, t_0))}\leq C r^{\alpha-1}, 
$$
with $C=C(n,a,\alpha)$.
Converting back to $W$, we have with $\tilde t_0= r^2 t_0$
\begin{equation}\label{eq:w-est}
\sup_{\Q_{r/4}((r/2)e_{n+1}, \tilde t_0)}|\nabla W|+r^\alpha [\nabla
W]_{H^{\alpha, \alpha/2}(\Q_{r/4}((r/2)e_{n+1}, \tilde t_0))}\leq C r^{\alpha}.
\end{equation}
Particularly, taking the $y$-component of $\nabla W$ only, we obtain
$$
\sup_{\Q_{r/4}((r/2)e_{n+1}, \tilde t_0)}|\partial_y V|+r^\alpha [\partial_y
V]_{H^{\alpha, \alpha/2}(\Q_{r/4}((r/2)e_{n+1}, \tilde t_0))}\leq C r^{\alpha}.
$$
If we now reposition the center of the ball to $(x,t)\in Q_{1/4}$, then for $|y|<1/4$ we will have
\begin{align}
  |\partial_y V|&\leq C |y|^\alpha,\label{eq:vy-est1}\\
  [\partial_yV]_{H^{\alpha, \alpha/2}(\Q_{|y|/2}(x,y,t))}&\leq C,\label{eq:vy-est2}
\end{align}
with $C=C(n,a,\alpha)$.
From here it is standard to conclude that
\begin{equation}\label{cor7}
\|\partial_y V\|_{H^{\alpha, \alpha/2}(\Q_{1/8})}\leq C.
\end{equation}
Indeed, let $(x^1,y^1, t^1)$ and $(x^2,y^2, t^2)$ be two points in $\Q_{1/8}$. Assume
$|y^1|\geq |y^2|$. Consider then two cases

1) $|(x^1,y^1,t^1)-(x^2,y^2,t^2)|\geq |y^1|/2$. Then by \eqref{eq:vy-est1}
\begin{align*}
|\partial_y V(x^1,y^1,t^1)-\partial_y V(x^2,y^2,t^2)|&\leq |\partial_y
                                         V(x^1,y^1,t^1)|+|\partial_y V(x^2,y^2,t^2)|\\
                                         & \leq 2 C(n,a,\alpha) |y^1|^\alpha\\
  &\leq C(n,a,\alpha) |(x^1,y^1,t^1)-(x^2,y^2,t^2)|^\alpha
\end{align*}

2) $|(x^1,y^1,t^1)-(x^2,y^2,t^2)|<|y^1|/2$. Then by \eqref{eq:vy-est2}
\begin{equation}\label{l1}
|\partial_y V(x^1,y^1,t^1)-\partial_y V(x^2,y^2,t^2)|\leq C(n,a,\alpha) |(x^1,y^1,t^1)-(x^2,y^2,t^2)|^\alpha.
\end{equation}
This proves \eqref{cor7}. It remains to show that $\nabla_x V\in
H^{\alpha, \alpha/2}(\Q_{1/8})$. By taking the $x$-components in
\eqref{eq:w-est}, we will have
$$
\sup_{\Q_{r/4}((r/2)e_{n+1}, -r^2/2)}|\nabla_x V(X,t)-\nabla_x V(0,0)|+r^\alpha
[\nabla_x V]_{H^{\alpha, \alpha/2}(\Q_{r/4}((r/2)e_{n+1}, -r^2/2))}\leq C r^{\alpha}
$$
and repositioning the origin to $(x,0,t)$ with $(x,t)\in Q_{1/4}$, we will
have that 
for any $|y|<1/4$

\begin{align}
  |\nabla_x V(x,y,t)-\nabla_x V(x,0,t)|&\leq C |y|^\alpha,\label{eq:vx-est1}\\
  [\nabla_xV(x,y,t)]_{\mathbb{H}^{\alpha}(\Q_{|y|/2}(x,y,t))}&\leq C.\label{eq:vx-est2}
\end{align}
In particular,
$$
\osc_{\Q_{|y|/2}(x,y,t)} \nabla_x V\leq C |y|^\alpha.
$$
Now, let $(x^1,t^1),(x^2,t^2)\in Q_{1/4}$ and $r=|(x^1,t^1)-(x^2,t^2)|$. Consider then two
points $(x^1,r)$ and $(x^2,r)$. Then from the oscillation estimate above
$$
|\nabla_x V(x^1,r, t^1)-\nabla_x V(x^2,r,t^2)|\leq
C r^\alpha
$$
and combined with \eqref{eq:vx-est1}
\begin{align*}
|\nabla_x
V(x^1,0, t^1)-\nabla_xV(x^2,0, t^2)|&\leq |\nabla_x V(x^1,0,t^1)-\nabla_x
  V(x^1,r,t^1)|\\
  &\qquad +|\nabla_x V(x^1,r, t^1)-\nabla_x V(x^2,r, t^2)|\\
  &\qquad+|\nabla_x
  V(x^2,r, t^2)-\nabla_x V(x^2,0, t^2)|\\&\leq C |(x^1, t^1)-(x^2, t^2)|^\alpha.
\end{align*}
Particularly, $\nabla_x V(\cdot,0)\in \mathbb{H}^{\alpha}(\Q_{1/4})$.
Now, considering the difference
$$
\nabla_xV(x,y,t)-\nabla_x V(x,0,t)
$$
and applying the same arguments as we did for $\partial_y V$, we can conclude that
$\nabla_xV-\nabla_x V(\cdot,0, \cdot)\in H^{\alpha, \alpha/2}(\Q_{1/8})$. Recalling also
that $\nabla_x V(\cdot,0, \cdot)\in H^{\alpha, \alpha/2}(Q_{1/4})$, we conclude that
$\nabla_x V\in H^{\alpha, \alpha/2}(\Q_{1/8})$, with bounds on the appropriate
$H^{\alpha, \alpha/2}$-norms depending only on $n$, $a$ and $\alpha$.
\end{proof}
We end this subsection with the following important remark.

\begin{rmrk}\label{conh}
With $V$ as in  Proposition~\ref{ac1}, for a given $\ve>0$ since 
\[
|V_y| \leq C(n, a, \ve) y^{1-\ve}
\]
therefore we have the following decay estimate for $y^a U_y$
\begin{equation}\label{dece}
|y^a V_y| \leq C(n, a, \ve) y^{1+a - \ve}
\end{equation}
From \eqref{dece}, it follows by arguing as in \eqref{eq:vy-est1}--\eqref{l1} that $y^a V_y$ is in  $H^{1+a-\ve, \frac{1+a-\ve}{2}}$ for a given $\ve>0$ up to $\{y=0\}$ and moreover an analogous estimate as in \eqref{K} holds for the corresponding H\"older norm. 
\end{rmrk}

\subsection{Regularity of odd solutions}
We now  establish regularity for odd solutions to $\La U= |y|^a f$ or equivalently  for solutions that vanish on $\{y=0\}$.

\begin{lemma}\label{odd}
Let $U$ be a solution to 
\begin{equation}
\begin{cases}
\La U= |y|^a f &\text{in $\Q_1^+$},
\\
U=0 &\text{on $\{y=0\}$},
\end{cases}
\end{equation}
where 
\begin{equation}\label{der}
\|f\|_{L^{\infty}(\Q_1^+)},\quad \|\nabla_x f\|_{L^{\infty}(\Q_1^+)} \leq K,\quad
|\partial_y f|  \leq Ky.
\end{equation}
Then, we have that  for some $\alpha \in (0,1)$, the following estimate holds,
\begin{equation}\label{ert}
\|\nabla_x U\|_{H^{\alpha, \frac{\alpha}{2}}(\overline{\Q_{1/2}^+})} +  \|y^a U_y\|_{H^{\alpha, \frac{\alpha}{2}}(\overline{\Q_{1/2}^+})} \leq C ( \|U\|_{L^{2}(\Q_1^+,|y|^adXdt)} + K).
\end{equation}
\end{lemma}

\begin{proof}
By an odd reflection, we note that $U$ solves a similar equation in $\Q_1$ with bounded right hand side. Therefore from the regularity result in \cite{CSe} it follows that $U \in H_{\rm loc}^{\alpha, \frac{\alpha}{2}}$ up to $\{y=0\}$. Now by taking repeated difference quotients of the type 
\[
U_{h, e_i}= \frac{U(x+he_i, y, t)- U(X,t)}{h^{k \alpha}}
\]
for $k=1, 2...$  and so on and by using zero Dirichlet conditions, the desired  estimate for $\nabla_x U$  follows by a repeated  application of  such a H\"older continuity result.  Now after an odd reflection,  we also have  that $ w= |y|^a U_y$ solves the following conjugate equation
\[
\mathscr{L}_{-a} w=  (\tilde f)_y
\]
in $\Q_1$ where $\tilde f$ is the odd extension of $f$ across $\{y=0\}$. Moreover by  using the second estimate in \eqref{der}, we observe that $w$ solves an equation of the type
\[
\mathscr{L}_{-a} w=  |y|^{-a} G
\]
where $G \in L^{\infty}$. Over here we also use the fact that $a \in (-1,1)$. Then again from \cite{CSe} it follows that $w$ is H\"older continuous in $\Q_{1/2}$ from the which the desired H\"older estimate for $y^a U_y$ follows. 
\end{proof}


\begin{thebibliography}{99}

\bibitem{Ath82}
I. Athanasopoulos, \emph{Regularity of the solution of an evolution problem with inequalities on the boundary}, Comm. Partial Differential Equations \textbf{7}~(1982), no.~12, 1453--1465
  


\bibitem{ACM}
I. Athanasopoulos, L. Caffarelli \& E. Milakis, \emph{On the
  regularity of the non-dynamic parabolic fractional obstacle
  problem},  J. Differential Equations \textbf{265}~(2018), no.~6, 2614--2647.

\bibitem{ACM1} I. Athanasopoulos, L. Caffarelli \& E. Milakis, \emph{Parabolic obstacle problems, quasi-convexity and regularity}, Ann. Sc. Norm. Super. Pisa Cl. Sci. (5) \textbf{XIX} (2019), 781--825.


\bibitem{ACS} 
I. Athanasopoulos, L. A. Caffarelli \& S. Salsa, \emph{The structure of the free boundary for lower dimensional obstacle problems}, Amer. J. Math. \textbf{130}~(2008), no.~2, 485--498. 


\bibitem{AU88}
A. A. Arkhipova and N. N. Uraltseva, \emph{Regularity of the solution of a
  problem with a two-sided limit on a boundary for elliptic and
  parabolic equations}, Trudy Mat. Inst. Steklov. \textbf{179} (1988),
5--22, 241 (Russian). Translated in Proc. Steklov Inst. Math. 1989, no.~2, 1–19; Boundary value problems of mathematical physics, 13.
 
\bibitem{BG}
A. Banerjee \& N. Garofalo, \emph{Monotonicity of generalized frequencies and the strong unique continuation property for fractional parabolic equations}. Adv. Math. \textbf{336}~(2018), 149--241. 

\bibitem{BDGP2}
A. Banerjee, D. Danielli, N. Garofalo \& A. Petrosyan, \emph{The structure of the singular set in the thin obstacle problem for degenerate parabolic equations}, arXiv:1902.07457

\bibitem{CDS}
L. A. Caffarelli, D. De Silva \& O. Savin,  \emph{The two membranes problem for different operators}. Ann. Inst. H. Poincar\'e Anal. Non Lin\'eaire \textbf{34}~(2017), no.~4, 899--932.



\bibitem{CSS}
L. A. Caffarelli, S. Salsa \& L. Silvestre,  \emph{Regularity estimates for the solution and the free boundary of the obstacle problem for the fractional Laplacian}, Invent. Math. \textbf{171}~(2008), no.~2, 425--461.



\bibitem{CSe}
F. Chiarenza \& R. Serapioni, \emph{A remark on a Harnack inequality for degenerate parabolic equations.}, Rend. Sem. Mat. Univ. Padova \textbf{73}~ (1985), 179--190.  

\bibitem{DGPT}
D. Danielli, N. Garofalo, A. Petrosyan \& T. To, \emph{Optimal regularity and the free boundary in the parabolic Signorini problem}, Mem. Amer. Math. Soc. 249 (2017), no.~1181, v + 103 pp. 

\bibitem{DL}
G. Duvaut \& J.-L. Lions, \emph{Les in\'equations en m\'ecanique et en physique}, (French) Travaux et Recherches Math\'ematiques, no.~21. Dunod, Paris, 1972. xx+387 pp.





\bibitem{GP}
N. Garofalo \& A. Petrosyan, \emph{Some new monotonicity formulas and the singular set in the lower dimensional obstacle problem}, Invent. Math. \textbf{177}~(2009). no.2, 415--461.

\bibitem{GPG}
N. Garofalo, A. Petrosyan \& M. Smit Vega Garcia, \emph{An epiperimetric inequality approach to the regularity of the free boundary in the Signorini problem with variable coefficients}. J. Math. Pures Appl. (9) \textbf{105}~ (2016), no.~6, 745--787.

\bibitem{GPPS}
N. Garofalo, A. Petrosyan, C. A. Pop \& M. Smit Vega Garcia,  \emph{Regularity of the free boundary for the obstacle problem for the fractional Laplacian with drift}, Ann. Inst. H. Poincar\'e Anal. Non Lin\'aire \textbf{34}~(2017), no.~3, 533--570. 



\bibitem{GRO}
N. Garofalo \& X. Ros-Oton, \emph{Structure and regularity of the singular set in the obstacle problem for the fractional Laplacian},  Revista Mat. Iberoamer., to appear.

\bibitem{GG}
N. Garofalo \& M. Smit Vega Garcia, \emph{New monotonicity formulas and the optimal regularity in the Signorini problem with variable coefficients},
Adv. Math. \textbf{262}~ (2014), 682--750.




\bibitem{Ne}
A. Nekvinda, \emph{Characterization of traces of the weighted Sobolev space $W^{1,p}(\Om,d^{\ve}M)$ on $M$}, Czechoslovak Math. J., \textbf{43}~ (1993), 695--711.

\bibitem{NS}
K. Nystr\"om \& O. Sande, \emph{Extension properties and boundary estimates for a fractional heat operator}, Nonlinear Analysis, \textbf{140}~(2016), 29--37.

\bibitem{PP}
A. Petrosyan \& C. A. Pop, \emph{Optimal regularity of solutions to the obstacle problem for the fractional Laplacian with drift}. J. Funct. Anal. \textbf{268}~(2015), no.~2, 417--472. 

   
\bibitem{PZ}
A. Petrosyan \& A. Zeller, \emph{Boundedness and continuity of the time derivative in the parabolic Signorini problem},  Math. Res. Let. \textbf{26}~(2019), no.~1, 281--292.
   





  

\bibitem{ST}
P. R. Stinga \& J. L. Torrea, \emph{Regularity theory and extension problem for fractional nonlocal parabolic equations and the master equation}, SIAM J. Math. Anal., \textbf{49}~(2017), no.~5, 3893--3924. 

\bibitem{STV}
Y. Sire, S. Terracini \& S. Vita, \emph{Liouville type theorems and
  regularity of solutions to degenerate or singular problems Part
  I:Even solutions}, arXiv:1904.02143v1.

\bibitem{Ura85} N. N. Uraltseva, \emph{H\"older continuity of gradients of solutions of parabolic equations with boundary conditions of Signorini type}, Dokl. Akad. Nauk SSSR \textbf{280}~(1985), no.~3, 563--565 (Russian).


\end{thebibliography}
\end{document}